\documentclass[11pt]{article}
\usepackage[letterpaper]{geometry}
\usepackage{amsmath,amsthm,amsfonts,amssymb}
\usepackage{enumerate,color,xcolor}
\usepackage{graphicx}
\usepackage{subfigure}
\usepackage{url}

\numberwithin{equation}{section}
\theoremstyle{plain}

\newtheorem{theorem}{Theorem}

\newtheorem{lemma}[theorem]{Lemma}

\newtheorem{proposition}[theorem]{Proposition}

\newtheorem{remark}[theorem]{Remark}

\begin{document}

\begin{center}
  \Large \bf Precise deviations for Hawkes processes
\end{center}

\author{}
\begin{center}
{Fuqing Gao}\,\footnote{School of Mathematics and Statistics, Wuhan University, Wuhan 430072, People's Republic of China; fqgao@whu.edu.cn},
  Lingjiong Zhu\,\footnote{Department of Mathematics, Florida State University, 1017 Academic Way, Tallahassee, FL-32306, United States of America; zhu@math.fsu.edu.
  }
\end{center}

\begin{center}
 \today
\end{center}

\begin{abstract}
Hawkes process is a class of simple point processes with self-exciting
and clustering properties. Hawkes process has been widely applied 
in finance, neuroscience, social networks, criminology, seismology, and many
other fields. In this paper, we study precise deviations
for Hawkes processes for large time asymptotics, that strictly extends and improves the
existing results in the literature. Numerical illustrations will also be provided.
\end{abstract}


\section{Introduction}

The Hawkes process is a simple point process $N_{t}$
with the stochastic intensity given by:
\begin{equation}\label{intensity}
\lambda_{t}=\nu+\int_{0}^{t-}h(t-s)dN_{s},
\end{equation}
where $\nu>0$ and $h:\mathbb{R}_{+}\rightarrow\mathbb{R}_{+}$
being locally bounded. We assume that $N_{0}=0$, that is, the Hawkes process
starts at time zero with empty history. 
The Hawkes process is named after Alan Hawkes \cite{Hawkes}.
In the literature $\nu$ is called the baseline intensity,
and $h$ is called the exciting function, or kernel function,
encoding the influence of past events on the intensity. 
Br\'{e}maud and Massouli\'{e} \cite{Bremaud} generalized
the dynamics \eqref{intensity} and the formula for the intensity \eqref{intensity}
by a nonlinear function of $\int_{0}^{t-}h(t-s)dN_{s}$,
and hence came the name nonlinear Hawkes processes.
The original model \eqref{intensity} proposed by Hawkes \cite{Hawkes}
is thus sometimes referred to as the linear Hawkes process.

For the Hawkes process \eqref{intensity}, the occurrence of a jump increases the intensity
of the point process, and thus increases the likelihood of more future jumps.
On the other hand, the intensity declines when there is no occurrence of new jumps. 
The self-exciting and clustering property makes the Hawkes process
very appealing in applications in finance and many other fields. 
The Hawkes process is widely used in the modeling of the limit order books in high frequency trading, see
e.g. Alfonsi and Blanc \cite{AB} for optimal execution, and Abergel and Jedidi \cite{AJ} for ergodicity in Hawkes based limit order books models, and also the modeling of the duration between trades, see e.g. Bauwens and Hautsch \cite{BH} or the arrival process of buy and sell orders, see e.g. Bacry et al. \cite{BacryQF}. 
The Hawkes process also finds applications in dark pool trading \cite{GZ4}.
In the context of credit risk modeling, 
Errais et al. \cite{Errais} used a top down approach using the affine point process,
which includes Hawkes process as a special case.
Giot \cite{Giot}, Chavez-Demoulin et al. \cite{Chavez} tested Hawkes processes in the risk management context.
A\"{i}t-Sahalia et al. \cite{Ait, AitII} used Hawkes processes to model two key aspects of asset prices: clustering
in time and cross sectional contamination between regions.
The Hawkes process 
has also been used to explain the supply and demand microstructure in an interest rate model in Hainaut \cite{Hainaut}.
The applications other than finance include: 
neuroscience, see e.g. \cite{PerniceI,Reynaud},
genome analysis, see e.g. \cite{Gusto, Reynaud}, networks and sociology, see e.g.
\cite{Crane,zhao2015seismic}, queueing theory, see e.g. \cite{GZ3}, insurance, see e.g. \cite{Stabile,ZhuRuin},
criminology, see e.g. \cite{Mohler,Porter},
seismology, see e.g. \cite{OgataII,OgataIII} and many other fields.

In this paper, we consider the linear Hawkes process $N_{t}$ 
with $N_{0}=0$ with the intensity \eqref{intensity}.
We assume throughout this paper that
\begin{equation}\label{assump:throughout}
\Vert h\Vert_{L^{1}}=\int_{0}^{\infty}h(t)dt<1,
\qquad\text{and}\qquad
\int_{0}^{\infty}th(t)dt<\infty.
\end{equation}

Let us first review the limit theorems for linear Hawkes processes in 
the literature. 
It is well known that 
under the assumption $\Vert h\Vert_{L^{1}}<1$, 
there exists a unique stationary Hawkes process defined
on $\mathbb{R}$ with the intensity
$\lambda_{t}=\nu+\int_{-\infty}^{t-}h(t-s)dN_{s}$,
and we have the law of large numbers
$\frac{N_{t}}{t}\rightarrow\frac{\nu}{1-\Vert h\Vert_{L^{1}}}$
as $t\rightarrow\infty$.
Bacry et al. \cite{Bacry} obtained a functional
central limit theorem for multivariate Hawkes process
and as a special case of their result,
\begin{equation}\label{eqn:CLT}
\frac{N_{t}-\frac{\nu}{1-\Vert h\Vert_{L^{1}}}t}{\sqrt{t}}\rightarrow N\left(0,\frac{\nu}{(1-\Vert h\Vert_{L^{1}})^{3}}\right),
\end{equation}
in distribution as $t\rightarrow\infty$
under the assumption that $\int_{0}^{\infty}t^{1/2}h(t)dt<\infty$.
For the central limit theorem \eqref{eqn:CLT}, also see Hawkes and Oakes \cite{HawkesII}.
Moreover, Bordenave and Torrisi \cite{Bordenave} proved that 
$\mathbb{P}(\frac{N_{t}}{t}\in\cdot)$
satisfies a large deviation principle, with the rate function:
\begin{equation}\label{Ifunction}
I(x)=
x\log\left(\frac{x}{\nu+x\Vert h\Vert_{L^{1}}}\right)-x+x\Vert h\Vert_{L^{1}}+\nu,
\end{equation}
if $x\geq 0$ and $I(x)=+\infty$ otherwise.
Note that the rate function in the paper \cite{Bordenave}
is written as the Legendre transform expression, and the formula \eqref{Ifunction}
is first mentioned in \cite{ZhuMDP}. Also notice that
in \cite{Bordenave}, the assumption $\int_{0}^{\infty}th(t)dt<\infty$
is needed, and indeed this assumption is not necessary, see e.g. \cite{KarabashZhu}.
A moderate deviation principle is obtained in \cite{ZhuMDP}
that fills in the gap between the central limit theorem and the large deviation principle.
Other works on the asymptotics of linear Hawkes processes, including
the nearly unstable Hawkes processes, that is, when the $\Vert h\Vert_{L^{1}}$
is close to $1$, see e.g. \cite{JR,JRII},
and the large initial intensity asymptotics for the Markovian case \cite{GZ,GZ2},
and the large baseline intensity asymptotics \cite{GZ3}.

For nonlinear Hawkes processes, \cite{ZhuCLT} studies the central limit theorem,
and \cite{ZhuII} obtains a process-level, i.e., level-3 large deviation principle,
and hence has the scalar large deviations as a by-product. 
An alternative expression for the rate function when the system is Markovian
is obtained in \cite{ZhuI}. Recently, Torrisi \cite{TorrisiI,TorrisiII} studies the rate of convergence
in the Gaussian and Poisson approximations of the simple point processes with stochastic intensity,
which includes as a special case, the nonlinear Hawkes process.

The large deviations \cite{Bordenave} and moderate deviations \cite{ZhuMDP} for linear Hawkes processes
are of the Donsker-Varadhan type, which only gives the leading order term, but not the 
higher order ones. More precise deviations are desirable in many applications,
which motivates us to study the precise deviations for linear Hawkes processes. 
In this paper, we will derive the precise deviations for linear Hawkes processes, using
the recent mod-$\phi$ convergence theory developed in \cite{modphi}. 
To apply the mod-$\phi$ convergence theory, two key steps are to show
the convergence of the moment generating function at a certain speed,
and verify the limit corresponds to an infinitely divisible distribution. 
The moment generating function for linear Hawkes processes has semi-explicit formula, 
and then the precise deviations results follow from the mod-$\phi$ convergence theory
after careful analysis and a series of propositions and lemmas. 
The paper is organized as follows. 
We will state the main results of our paper 
in Section \ref{MainSection}.
In particular, we will give precise large deviations results
in Section \ref{LDPSection}
and precise moderate deviations and fluctuation results
in Section \ref{MDPSection}.
We provide numerical illustrations in Section \ref{sec:numerics}.
All the proofs will be given in Section \ref{ProofsSection}.

\section{Main Results}\label{MainSection}

In this paper, we apply the recently developed mod-$\phi$ convergence
method \cite{modphi} to obtain precise large deviations for linear Hawkes processes 
for the large time asymptotic regime. We will also obtain the precise moderate deviations
and some fluctuations results.

Let us first recall the definition of mod-$\phi$ convergence, see e.g. Definition 1.1.1. in \cite{modphi}.
Let $(X_{n})_{n\in\mathbb{N}}$ be a sequence of real-valued random variables
and $\mathbb{E}[e^{zX_{n}}]$ exist in a strip $\mathcal{S}_{(c,d)}:=\{z\in\mathbb{C}: c<\mathcal{R}(z)<d\}$,
with $c<0<d$ extended real numbers, i.e. we allow $c=-\infty$ and $d=+\infty$
and $\mathcal{R}(z)$ denotes the real part of $z\in\mathbb{C}$ throughout this paper.
We assume that there exists a non-constant infinitely divisible distribution $\phi$
with $\int_{\mathbb{R}}e^{zx}\phi(dx)=e^{\eta(z)}$, which is well defined
on $\mathcal{S}_{(c,d)}$, and an analytic function $\psi(z)$
that does not vanish on the real part of $\mathcal{S}_{(c,d)}$
such that locally uniformly in $z\in\mathcal{S}_{(c,d)}$,
\begin{equation}\label{eqn:mod:phi}
e^{-t_{n}\eta(z)}\mathbb{E}[e^{zX_{n}}]\rightarrow\psi(z),
\end{equation}
where $t_{n}\rightarrow+\infty$ as $n\rightarrow\infty$. 
Then we say that $X_{n}$ converges mod-$\phi$ on $\mathcal{S}_{(c,d)}$
with parameters $(t_{n})_{n\in\mathbb{N}}$ and limiting function $\psi$.

Assume that $\phi$ is a lattice distribution, 
i.e., a distribution with support included in $\gamma+\lambda\mathbb{Z}$ for some 
constants $\gamma,\lambda>0$. 
Also assume that the sequence of random variables $(X_{n})_{n\in\mathbb{N}}$
converges mod-$\phi$ at speed $O((t_{n})^{-v})$ (Definition 2.1.1. in \cite{modphi}), that is,
\begin{equation}
\sup_{z\in K}\left|e^{-t_{n}\eta(z)}\mathbb{E}[e^{zX_{n}}]-\psi(z)\right|\leq C_{K}(t_{n})^{-v},
\end{equation}
where $C_{K}>0$ is some constant, for any compact set $K\subset\mathcal{S}_{(c,d)}$.
Then Theorem 3.2.2. in \cite{modphi}
states that for any $x\in\mathbb{R}$ in the interval $(\eta'(c),\eta'(d))$
such that $t_{n}x\in\mathbb{N}$, we have
\begin{equation}\label{precise:1}
\mathbb{P}(X_{n}=t_{n}x)
=\frac{e^{-t_{n}F(x)}}{\sqrt{2\pi t_{n}\eta''(\theta^{\ast})}}
\left(\psi(\theta^{\ast})+\frac{a_{1}}{t_{n}}+\frac{a_{2}}{t_{n}^{2}}
+\cdots+\frac{a_{v-1}}{t_{n}^{v-1}}
+O\left(\frac{1}{t_{n}^{v}}\right)\right),
\end{equation}
as $n\rightarrow\infty$, where 
$\theta^{\ast}$ is defined via $\eta'(\theta^{\ast})=x$,
and $F(x):=\sup_{\theta\in\mathbb{R}}\{\theta x-\eta(\theta)\}$
is the Legendre transform of $\eta(\cdot)$, 
and if $x\in\mathbb{R}$ and $x\in(\eta'(0),\eta'(d))$, 
then, as $n\rightarrow\infty$,
\begin{equation}\label{precise:2}
\mathbb{P}(X_{n}\geq t_{n}x)
=\frac{e^{-t_{n}F(x)}}{\sqrt{2\pi t_{n}\eta''(\theta^{\ast})}}\frac{1}{1-e^{-\theta^{\ast}}}
\left(\psi(\theta^{\ast})+\frac{b_{1}}{t_{n}}+\frac{b_{2}}{t_{n}^{2}}
+\cdots+\frac{b_{v-1}}{t_{n}^{v-1}}
+O\left(\frac{1}{t_{n}^{v}}\right)\right),
\end{equation}
where $(a_{k})_{k=1}^{\infty}$, $(b_{k})_{k=1}^{\infty}$
are rational fractions in the derivatives of $\eta$ and $\psi$ at $\theta^{\ast}$.

The mod-$\phi$ convergence developed in \cite{modphi}
also implies the precise moderate deviations and extended central limit theorem,
see Theorem 3.3.1. and Corollary 3.3.5. in \cite{modphi}.
More precisely, the second part of Corollary 3.3.5. in \cite{modphi}
states that assuming the mod-$\phi$ convergence holds, then
for any $y=o((t_{n})^{1/2-1/m})$, where $m\geq 3$, $m\in\mathbb{N}$,
one has
\begin{equation}\label{precise:3}
\mathbb{P}\left(X_{n}\geq t_{n}\eta'(0)+\sqrt{t_{n}\eta''(0)}y\right)
=\frac{(1+o(1))}{y\sqrt{2\pi}}
\cdot e^{-\sum_{i=2}^{m-1}\frac{F^{(i)}(\eta'(0))}{i!}
\frac{(\eta''(0))^{i/2}y^{i}}{(t_{n})^{(i-2)/2}}}.
\end{equation}
The precise moderate deviations and extended central limit theorem
then follow from \eqref{precise:3}. See \cite{modphi} for further discussions.
Note that the precise deviations in \eqref{precise:1} and \eqref{precise:2}
require mod-$\phi$ convergence at speed $O((t_{n})^{-v})$, 
while the fluctuations result in \eqref{precise:3} only requires mod-$\phi$ convergence.

Finally, \cite{modphi} suggested in their Remark 3.2.5. how $(a_{k})_{k=1}^{\infty}$, $(b_{k})_{k=1}^{\infty}$
in \eqref{precise:1} and \eqref{precise:2}
can be computed without providing detailed formulas.
For the sake of applications and implementations, we will spell out 
the explicit formulas for $(a_{k})_{k=1}^{\infty}$ and $(b_{k})_{k=1}^{\infty}$
in the next proposition and the proof will be provided in the Appendix.

\begin{proposition}\label{prop:ab}
For any function $f$, denote its $k$-th order derivative by $f^{(k)}$.
Let $\mathcal{S}_{n}$ be the set consisting of 
all the $n$-tuples of non-negative integers $(m_{1},\ldots,m_{n})$
satisfying:
$1\cdot m_{1}+2\cdot m_{2}+3\cdot m_{3}+\cdots+n\cdot m_{n}=n$.

(i) For every $k\geq 1$,
\begin{align}
a_{k}&=\sum_{\ell=0}^{2k}\frac{\psi^{(2k-\ell)}(\theta^{\ast})}{(2k-\ell)!}
\sum_{\mathcal{S}_{\ell}}\frac{(-1)^{m_{1}+\cdots+m_{\ell}}}{m_{1}!1!^{m_{1}}m_{2}!2!^{m_{2}}\cdots m_{\ell}!\ell!^{m_{\ell}}}
\nonumber
\\
&\qquad\qquad
\cdot\prod_{j=1}^{\ell}\left(\frac{1}{\eta''(\theta^{\ast})}\frac{\eta^{(j+2)}(\theta^{\ast})}{(j+2)(j+1)}\right)^{m_{j}}
\frac{(-1)^{k}(2(k+m_{1}+\cdots+m_{\ell})-1)!!}{(\eta''(\theta^{\ast}))^{k}}.
\label{a:eqn}
\end{align}

(ii) For every $k\geq 1$,
\begin{align}
b_{k}&=
\sum_{n=0}^{2k}
\sum_{\mathcal{S}_{n}}\frac{e^{-(m_{1}+\cdots+m_{n})\theta^{\ast}}(m_{1}+\cdots+m_{n})!(1-e^{-\theta^{\ast}})^{-(m_{1}+\cdots+m_{n})-1}}
{m_{1}!1!^{m_{1}}m_{2}!2!^{m_{2}}\cdots m_{n}!n!^{m_{n}}}
\cdot
\prod_{j=1}^{n}(-1)^{j\cdot m_{j}}
\nonumber
\\
&\qquad
\cdot
\sum_{\ell=0}^{2k-n}\frac{\psi^{(2k-n-\ell)}(\theta^{\ast})}{(2k-n-\ell)!}
\sum_{\mathcal{S}_{\ell}}\frac{(-1)^{m_{1}+\cdots+m_{\ell}}}{m_{1}!1!^{m_{1}}m_{2}!2!^{m_{2}}\cdots m_{\ell}!\ell!^{m_{\ell}}}
\nonumber
\\
&\qquad\qquad
\cdot\prod_{j=1}^{\ell}\left(\frac{1}{\eta''(\theta^{\ast})}\frac{\eta^{(j+2)}(\theta^{\ast})}{(j+2)(j+1)}\right)^{m_{j}}
\frac{(-1)^{k}(2(k+m_{1}+\cdots+m_{\ell})-1)!!}{(\eta''(\theta^{\ast}))^{k}}.
\label{b:eqn}
\end{align}
\end{proposition}

\subsection{Precise Large Deviations}\label{LDPSection}

Our main results for the precise large deviations for the Hawkes process
is stated as follows. It provides the full expansion to arbitrary order 
in the large time asymptotic regime, which generalizes the large deviations
result in \cite{Bordenave}.

\begin{theorem}\label{LDPThm}
Given $v\in\mathbb{N}$.
Assume \eqref{assump:throughout} and the following condition hold:
\begin{equation}
\int_{0}^{\infty}s^{v+1}h(s)ds<\infty.
\end{equation}

(i) For any $x>0$, and $tx\in\mathbb{N}$, as $t\rightarrow\infty$,
\begin{equation}
\mathbb{P}(N_{t}=tx)
=e^{-tI(x)}\sqrt{\frac{I''(x)}{2\pi t}}
\left(\psi(\theta^{\ast})+\frac{a_{1}}{t}+\frac{a_{2}}{t^{2}}
+\cdots+\frac{a_{v-1}}{t^{v-1}}
+O\left(\frac{1}{t^{v}}\right)\right),
\end{equation}
where for any $\mathcal{R}(z)\leq\Vert h\Vert_{L^{1}}-1-\log\Vert h\Vert_{L^{1}}$,
\begin{equation}\label{defn:psi:varphi}
\psi(z):=e^{\nu \varphi(z)}, \quad \text{and}\quad  \varphi(z):=\int_{0}^{\infty}[F(s;z)-x(z)]ds,
\end{equation}
which is analytic in $\mathcal{S}_{(-\infty,\Vert h\Vert_{L^{1}}-1-\log\Vert h\Vert_{L^{1}})}$, 
and $F$ is the unique solution that satisfies
\begin{equation}\label{FEqn}
F(t;z)=e^{z+\int_{0}^{t}(F(t-s;z)-1)h(s)ds},
\end{equation}
with the constraint $|F(t;z)|\leq\frac{1}{\Vert h\Vert_{L^{1}}}$, and it is analytic in 
$\mathcal{S}_{(-\infty,\Vert h\Vert_{L^{1}}-1-\log\Vert h\Vert_{L^{1}})}$,
and $x(z):=F(\infty;z)$ exists and it satisfies the equation
\begin{equation}\label{xEqn}
x(z)=e^{z+\Vert h\Vert_{L^{1}}(x(z)-1)},
\end{equation}
and it is analytic in $\mathcal{S}_{(-\infty,\Vert h\Vert_{L^{1}}-1-\log\Vert h\Vert_{L^{1}})}$.
And $I(x)$ is defined in \eqref{Ifunction}, $I''(x)=\frac{\nu^{2}}{x(\nu+\Vert h\Vert_{L^{1}}x)^{2}}$, 
and 
\begin{equation}
\theta^{\ast}=\log\left(\frac{x}{\nu+\Vert h\Vert_{L^{1}}x}\right)
-\frac{\Vert h\Vert_{L^{1}}x}{\nu+\Vert h\Vert_{L^{1}}x}+\Vert h\Vert_{L^{1}},
\end{equation}
where $(a_{k})_{k=1}^{\infty}$
are rational fractions in the derivatives of $\eta$ and $\psi$ at $\theta^{\ast}$
whose formulas are given in \eqref{a:eqn}, 
where $\eta(z):=\nu(x(z)-1)$.

(ii) For any $x>\frac{\nu}{1-\Vert h\Vert_{L^{1}}}$ and $tx\in\mathbb{N}$, as $t\rightarrow\infty$,
\begin{equation}
\mathbb{P}(N_{t}\geq tx)
=e^{-tI(x)}\sqrt{\frac{I''(x)}{2\pi t}}\frac{1}{1-e^{-\theta^{\ast}}}
\left(\psi(\theta^{\ast})+\frac{b_{1}}{t}+\frac{b_{2}}{t^{2}}
+\cdots+\frac{b_{v-1}}{t^{v-1}}
+O\left(\frac{1}{t^{v}}\right)\right),
\end{equation}
where $(b_{k})_{k=1}^{\infty}$
are rational fractions in the derivatives of $\eta$ and $\psi$ at $\theta^{\ast}$
whose formulas are given in \eqref{b:eqn}.
\end{theorem}

Note that the mod-$\phi$ convergence theory developed in \cite{modphi}
works for the a sequence of random variables $X_{n}$ associated
with discrete times $t_{n}$, see \eqref{eqn:mod:phi}.
In Theorem \ref{LDPThm}, although $N_{t}$ is a continuous-time process,
the restriction $tx\in\mathbb{N}$ guarantees that $t$ takes
values along $t_{n}:=\frac{n}{x}$ so that $t$ takes discrete values
and we can define $X_{n}:=N_{n/x}$ to apply the mod-$\phi$ convergence theory developed 
for discrete times in \cite{modphi}.

In order to apply and implement Theorem \ref{LDPThm}, 
we need to compute the sequences $(a_{k})_{k=1}^{\infty}$ and
$(b_{k})_{k=1}^{\infty}$ whose formulas are provided in Proposition \ref{prop:ab}
which rely on the derivatives of $\eta$ and $\psi$ at $\theta^{\ast}$.
Next, we provide recursive formulas to compute the derivatives of
$\eta$ and $\psi$ at $\theta^{\ast}$ of any order.

\begin{proposition}\label{prop:etapsi}
(i) $\eta(\theta^{\ast})=\nu(x(\theta^{\ast})-1)$, and
for $k\geq 1$, $\eta^{(k)}(\theta^{\ast})=\nu x^{(k)}(\theta^{\ast})$.
For $k\geq 1$, $x^{(k)}(\theta^{\ast})$ can be computed recursively as:
\begin{align}
x^{(k)}(\theta^{\ast})
&=\frac{x(\theta^{\ast})}{1-\Vert h\Vert_{L^{1}}x(\theta^{\ast})}
\sum_{\mathcal{T}_{k}}\frac{k!\cdot\Vert h\Vert_{L^{1}}^{m_{1}+\cdots+m_{k-1}}}{m_{1}!1!^{m_{1}}m_{2}!2!^{m_{2}}\cdots m_{k-1}!(k-1)!^{m_{k-1}}}
\cdot
\prod_{j=1}^{k-1}(x^{(j)}(\theta^{\ast}))^{m_{j}}
\nonumber
\\
&\qquad
+\frac{x(\theta^{\ast})}{1-\Vert h\Vert_{L^{1}}x(\theta^{\ast})}\sum_{\ell=0}^{k-1}\binom{k}{\ell}
\sum_{\mathcal{S}_{\ell}}\frac{\ell!\cdot\Vert h\Vert_{L^{1}}^{m_{1}+\cdots+m_{\ell}}}{m_{1}!1!^{m_{1}}m_{2}!2!^{m_{2}}\cdots m_{\ell}!\ell!^{m_{\ell}}}
\cdot
\prod_{j=1}^{\ell}(x^{(j)}(\theta^{\ast}))^{m_{j}},
\nonumber
\end{align}
where $\mathcal{T}_{k}$ denotes the set of $(k-1)$-tuples
of non-negative integers $(m_{1},\ldots,m_{k-1})$ satisfying the constraint
$1\cdot m_{1}+2\cdot m_{2}+3\cdot m_{3}+\cdots+(k-1)\cdot m_{k-1}=k$.

(ii) For every $k\geq 1$,
\begin{equation}
\psi^{(k)}(\theta^{\ast})
=\sum_{\mathcal{S}_{k}}\frac{k!\cdot\nu^{m_{1}+\cdots+m_{k}}\cdot\psi(\theta^{\ast})}
{m_{1}!1!^{m_{1}}m_{2}!2!^{m_{2}}\cdots m_{k}!k!^{m_{k}}}
\cdot
\prod_{j=1}^{k}\left(\int_{0}^{\infty}\left[F^{(j)}(s;\theta^{\ast})-x^{(j)}(\theta^{\ast})\right]ds\right)^{m_{j}},
\end{equation}
where $F^{(k)}(\cdot;\theta^{\ast})$, $k\geq 1$, can be computed recursively as
\begin{align}\label{F:recursive}
&F^{(k)}(t;\theta^{\ast})
\\
&=F(t;\theta^{\ast})\cdot
\int_{0}^{t}F^{(k)}(t-s;\theta^{\ast})h(s)ds
\nonumber
\\
&
\quad
+F(t;\theta^{\ast})\cdot
\sum_{\mathcal{T}_{k}}\frac{k!}
{m_{1}!1!^{m_{1}}\cdots m_{k-1}!(k-1)!^{m_{k-1}}}
\prod_{j=1}^{k-1}\left(\int_{0}^{t}F^{(j)}(t-s;\theta^{\ast})h(s)ds\right)^{m_{j}}
\nonumber
\\
&
\quad
+F(t;\theta^{\ast})\cdot\sum_{\ell=0}^{k-1}\binom{k}{\ell}
\sum_{\mathcal{S}_{\ell}}\frac{\ell!}
{m_{1}!1!^{m_{1}}\cdots m_{\ell}!\ell!^{m_{\ell}}}
\prod_{j=1}^{\ell}\left(\int_{0}^{t}F^{(j)}(t-s;\theta^{\ast})h(s)ds\right)^{m_{j}}.
\nonumber
\end{align}
\end{proposition}

The main strategy of the proof of Theorem \ref{LDPThm}
is to show the mod-$\phi$ convergence and apply Theorem 3.2.2. in \cite{modphi}.
The proof of Theorem \ref{LDPThm} relies on a series
of lemmas and propositions that we will state later.
We will first recall and discuss some well-known properties for the linear Hawkes
process that will be used extensively in our proofs later.

Hawkes and Oakes \cite{HawkesII} first discovered that a linear Hawkes process has an immigration-birth representation.
The immigrants (roots) arrive according to a standard Poisson process $\bar{N}$ with intensity $\nu>0$.
Each immigrant generates children according to a Galton-Watson tree.
Indeed, for an immigrant that arrives at time $s$, children are generated
according to a Poisson process with intensity $h(t-s)$ at time $t\geq s$, 
and each child will generate children independently following the same mechanism.
It turns out that the total number of children generated by the immigrant
follows a Poisson distribution with parameter $\Vert h\Vert_{L^{1}}$
and conditional on the total number of children generated by the immigrant, say $K$,
then the birth times of the children subtracting the arrival time
of the immigrant are distributed as the order statistics of $k$ i.i.d. 
random variables with density $h(\cdot)/\Vert h\Vert_{L^{1}}$.
Finally, the Hawkes process $N_{t}$ is the number of all the immigrants and their descendants
that arrive on the time interval $[0,t]$.

By using the immigration-birth representation for linear Hawkes processes, 
it is known that (see e.g. \cite{ZhuMDP,KarabashZhu,GZ4})
\begin{equation}\label{NEqn}
\mathbb{E}[e^{zN_{t}}]=e^{\nu\int_{0}^{t}(F(s;z)-1)ds},
\end{equation}
where $F$ satisfies the equation:
\begin{equation}
F(t;z)=e^{z+\int_{0}^{t}(F(t-s;z)-1)h(s)ds},\qquad t\geq 0,
\end{equation}
which holds for any $z\in\mathbb{C}$, such that 
\begin{equation}\label{thetacEqn}
\mathcal{R}(z)\leq\theta_{c}:=\Vert h\Vert_{L^{1}}-1-\log\Vert h\Vert_{L^{1}}.
\end{equation}

Note that by the immigration-birth representation, we can interpret $F(t;z)$ as
$F(t;z)=\mathbb{E}[e^{zS_{t}}]$,
where $S_{t}$ is the number of all the descendants of an immigrant
that arrives at time $0$, on the time interval $[0,t]$ including the immigrant.
Moreover, let us define
$x(z)=\mathbb{E}[e^{zS_{\infty}}]$.
It is well known that $x(z)$ satisfies the algebraic equation, see e.g. \cite{KarabashZhu}:
\begin{equation}
x(z)=e^{z+\Vert h\Vert_{L^{1}}(x(z)-1)}.
\end{equation}
This algebraic equation may have more than one solution.
It is known that for $z\in\mathbb{R}$, 
there are at most two solutions of this algebraic equation 
and $\mathbb{E}[e^{zS_{\infty}}]$ is the smaller solution, see e.g. \cite{KarabashZhu}.
By dominated convergence theorem, for $\mathcal{R}(z)\leq\theta_{c}$, where $\theta_{c}$ is defined in \eqref{thetacEqn},
we have $F(t,z)\rightarrow x(z)$,
as $t\rightarrow\infty$.
The limit $x(\cdot)$ has the following properties (Part (ii) of the following Proposition \ref{xprop} follows
from Theorem 3.2.1. in \cite{Bordenave}):

\begin{proposition}\label{xprop}
For any $\theta\in\mathbb{R}$, and $\theta\leq\theta_{c}$, where $\theta_{c}$ is defined in \eqref{thetacEqn},
we have

(i) $x(\theta)\Vert h\Vert_{L^{1}}\leq 1$.
(ii) $x'(\theta)\rightarrow\infty$ as $\theta\uparrow\theta_{c}$.
\end{proposition}

We know that $\mathbb{E}[e^{zS_{t}}]$
satisfies the equation \eqref{FEqn}.
Thus, as a by-product of the immigration-birth representation
for linear Hawkes processes, we get
the existence of the solution of \eqref{FEqn}.
Let us notice that for any $\mathcal{R}(z)\leq\theta_{c}$,
\begin{equation}
|F(t;z)|=\left|\mathbb{E}\left[e^{zS_{t}}\right]\right|
\leq\mathbb{E}\left[\left|e^{zS_{t}}\right|\right]
=\mathbb{E}\left[e^{\mathcal{R}(z)S_{t}}\right]
\leq\frac{1}{\Vert h\Vert_{L^{1}}}.
\end{equation}
Therefore, it suffices to consider the solution of the equation \eqref{FEqn}
that satisfies the constraint $|F(t;z)|\leq\frac{1}{\Vert h\Vert_{L^{1}}}$.
With this additional constraint, the equation \eqref{FEqn} has a unique solution:

\begin{proposition}\label{uniqueProp}
Let $z\in\mathbb{C}$ and $\mathcal{R}(z)\leq\theta_{c}$. 
The equation \eqref{FEqn} with the constraint $|F(t;z)|\leq\frac{1}{\Vert h\Vert_{L^{1}}}$
has a unique solution.
\end{proposition}

The key to prove the main result Theorem \ref{LDPThm} is
to verify the mod-$\phi$ convergence. More precisely, we need to show that
\begin{equation}
e^{-t\eta(z)}\mathbb{E}[e^{zN_{t}}]
\rightarrow\psi(z):=e^{\nu \varphi(z)},
\end{equation}
as $t\rightarrow\infty$ locally uniformly in $z$
for $\mathcal{R}(z)<\Vert h\Vert_{L^{1}}-1-\log\Vert h\Vert_{L^{1}}$, 
where 
$\varphi(z)=\int_{0}^{\infty}[F(s;z)-x(z)]ds$,
is analytic in $z$ and 
$\eta(z)=\nu(x(z)-1)$,
and
$e^{\eta(z)}=e^{\nu(x(z)-1)}=\mathbb{E}[e^{zY}]$,
for some random variable $Y$, where $Y$ has an infinitely divisible distribution.
We will show the mod-$\phi$ convergence below via a series of lemmas.

First, we show that $Y$ is infinitely divisible. The infinite divisibility
is a limitation of the method of mod-$\phi$ convergence. Fortunately,
the limiting distribution in the case of the linear Hawkes process is indeed infinitely divisible.

\begin{lemma}\label{divLemma}
$Y$ has an infinitely divisible distribution.
\end{lemma}

To show the mod-$\phi$ convergence, the main technical lemma
is given as follows:

\begin{lemma}\label{MainLemma}
For any $\mathcal{R}(z)<\theta_{c}$, where $\theta_{c}$ is defined in \eqref{thetacEqn}, 
\begin{equation}
\varphi(z)=\int_{0}^{\infty}[F(s;z)-x(z)]ds
\end{equation}
is well-defined and analytic, and as $t\rightarrow\infty$,
\begin{equation}
e^{-t(\nu(x(z)-1))}\mathbb{E}[e^{zN_{t}}]
\rightarrow e^{\nu \varphi(z)},
\end{equation}
locally uniformly in $z$. 
In addition, if $\int_{0}^{\infty}s^{v+1}h(s)ds<\infty$, then 
for any compact set $K$, there exists some $C_{K}>0$ such that
$\sup_{z\in K}|e^{-t(\nu(x(z)-1))}\mathbb{E}[e^{zN_{t}}]
-e^{\nu \varphi(z)}|\leq C_{K}t^{-v}$.
\end{lemma}

To this end, we have established the mod-$\phi$ convergence
for the linear Hawkes process for the large time limit. 
The proofs of all the propositions, lemmas and Theorem \ref{LDPSection}
will be given in Section \ref{ProofsSection}.


\subsection{Precise Moderate Deviations and Fluctuations}\label{MDPSection}

The mod-$\phi$ convergence also implies the precise moderate deviations and extended central limit theorem,
see Theorem 3.3.1. and Corollary 3.3.5. in \cite{modphi}.
By using Corollary 3.3.5. in \cite{modphi}, which is re-stated in our \eqref{precise:3}, 
we have the following result:

\begin{theorem}\label{MDPThm}
Assume \eqref{assump:throughout} holds. 
If $y=o(t^{1/2-1/m})$, where $m\geq 3$, then as $t\rightarrow\infty$,
\begin{equation}
\mathbb{P}\left(N_{t}\geq\frac{\nu}{1-\Vert h\Vert_{L^{1}}}t
+\sqrt{t}\frac{\sqrt{\nu}}{(1-\Vert h\Vert_{L^{1}})^{3/2}}y\right)
=\frac{(1+o(1))}{y\sqrt{2\pi}}
e^{-\sum_{i=2}^{m-1}\frac{I^{(i)}(\eta'(0))}{i!}\frac{(\eta''(0))^{i/2}y^{i}}{t^{(i-2)/2}}},
\end{equation}
where $I(\cdot)$ is defined in \eqref{Ifunction} and for any $i\geq 2$,
\begin{equation}\label{Ii}
I^{(i)}(x)=(i-2)!(-1)^{i-2}x^{1-i}
\left((i-1)\left(\frac{\Vert h\Vert_{L^{1}}x}{\nu+\Vert h\Vert_{L^{1}}x}\right)^{i}
-i\left(\frac{\Vert h\Vert_{L^{1}}x}{\nu+\Vert h\Vert_{L^{1}}x}\right)^{i-1}+1\right),
\end{equation}
and $\eta'(0)=\frac{\nu}{1-\Vert h\Vert_{L^{1}}}$,
and $\eta''(0)=\frac{\nu}{(1-\Vert h\Vert_{L^{1}})^{3}}$.
\end{theorem}

\begin{remark}
By letting $m=3$ in Theorem \ref{MDPThm} and $I''(\eta'(0))=1/\eta''(0)$, we get an extended
central limit theorem, that is, for any $y=o(t^{1/6})$, as $t\rightarrow\infty$,
\begin{equation}\label{extended:CLT}
\mathbb{P}\left(N_{t}\geq\frac{\nu}{1-\Vert h\Vert_{L^{1}}}t
+\sqrt{t}\frac{\sqrt{\nu}}{(1-\Vert h\Vert_{L^{1}})^{3/2}}y\right)
=\overline{\Phi}(y)(1+o(1)),
\end{equation}
where  $\overline{\Phi}(y):=\int_{y}^{\infty}\frac{1}{\sqrt{2\pi}}e^{-x^{2}/2}dx$.
Note that \eqref{extended:CLT} extends the univariate case of the Hawkes process central
limit theorem considered in \cite{Bacry} since we allow $y=o(t^{1/6})$.
\end{remark}

\begin{remark}
By Theorem 3.3.1. in \cite{modphi}, we can also study the moderate deviations result. 
If $1\ll y\ll\sqrt{t}$ for $t\rightarrow\infty$
\footnote{For any two functions $f$ and $g$, $f(t)\ll g(t)$ as $t\rightarrow\infty$ means
$\frac{f(t)}{g(t)}\rightarrow 0$ as $t\rightarrow\infty$.}, i.e., the moderate deviations regime, 
and if we let: 
\begin{equation}
s_{t}:=\frac{\nu}{1-\Vert h\Vert_{L^{1}}}
+\frac{\sqrt{\nu}}{(1-\Vert h\Vert_{L^{1}})^{3/2}}\frac{y}{\sqrt{t}},
\end{equation} 
then, with $\eta'(\theta^{\ast})=s_{t}$, as $t\rightarrow\infty$,
\begin{equation}
\mathbb{P}\left(N_{t}\geq\frac{\nu}{1-\Vert h\Vert_{L^{1}}}t
+\sqrt{t}\frac{\sqrt{\nu}}{(1-\Vert h\Vert_{L^{1}})^{3/2}}y\right)
=\frac{e^{-tI(s_{t})}}{\theta^{\ast}\sqrt{2\pi t\eta''(\theta^{\ast})}}(1+o(1)).
\end{equation}
\end{remark}

\begin{remark}\label{m:4:remark}
By letting $m=4$ in Theorem \ref{MDPThm} and $I''(\eta'(0))=1/\eta''(0)$ and $I'''(\eta'(0))=-\eta'''(0)/(\eta''(0))^{3}$, we get 
a precise moderate deviation result, that is,
for any $y=o(t^{1/4})$, as $t\rightarrow\infty$,
\begin{equation}\label{precise:MDP}
\mathbb{P}\left(N_{t}\geq\frac{\nu}{1-\Vert h\Vert_{L^{1}}}
+\sqrt{t}\frac{\sqrt{\nu}}{(1-\Vert h\Vert_{L^{1}})^{3/2}}y\right)
=\frac{(1+o(1))}{y\sqrt{2\pi}}e^{-\frac{y^{2}}{2}}
e^{\frac{\eta'''(0)}{6(\eta''(0))^{3/2}}\frac{y^{3}}{\sqrt{t}}}.
\end{equation}
Note that \eqref{MDPThm} gives a precise moderate deviation result,
and it provides a more precise tail estimate than \cite{ZhuMDP}.
To see this, let $y=\frac{(1-\Vert h\Vert_{L^{1}})^{3/2}}{\sqrt{\nu}}\frac{a(t)}{\sqrt{t}}x$, 
where $x$ is a constant independent of $t$.
Then for $t^{1/2}\ll a(t)\ll t^{3/4}$, as $t\rightarrow\infty$, we have
\begin{align*}
&\mathbb{P}\left(N_{t}\geq\frac{\nu}{1-\Vert h\Vert_{L^{1}}}t
+a(t)x\right)
\\
&=\frac{\sqrt{\nu}(1+o(1))}{(1-\Vert h\Vert_{L^{1}})^{3/2}\sqrt{2\pi}}\frac{\sqrt{t}}{a(t)x}e^{-\frac{(1-\Vert h\Vert_{L^{1}})^{3}}{2\nu}
\frac{a(t)^{2}}{t}x^{2}}
e^{\frac{\eta'''(0)}{6(\eta''(0))^{3}}\frac{a(t)^{3}}{t^{2}}x^{3}}
\\
&=\frac{\sqrt{\nu}(1+o(1))}{(1-\Vert h\Vert_{L^{1}})^{3/2}\sqrt{2\pi}}\frac{\sqrt{t}}{a(t)x}e^{-\frac{(1-\Vert h\Vert_{L^{1}})^{3}}{2\nu}
\frac{a(t)^{2}}{t}x^{2}}
e^{\frac{(1+2\Vert h\Vert_{L^{1}})(1-\Vert h\Vert_{L^{1}})^{4}}{6\nu^{2}}\frac{a(t)^{3}}{t^{2}}x^{3}},
\end{align*}
where $\eta'(0)=\frac{\nu}{1-\Vert h\Vert_{L^{1}}}$,
$\eta''(0)=\frac{\nu}{(1-\Vert h\Vert_{L^{1}})^{3}}$,
and
$\eta'''(0)=\nu\frac{1+2\Vert h\Vert_{L^{1}}}{(1-\Vert h\Vert_{L^{1}})^{5}}$
\footnote{The derivation of $\eta'''(0)$ will be given in the proof of Theorem \ref{MDPThm}.}.
\end{remark}

\section{Numerical Illustrations}\label{sec:numerics}

In this section, we illustrate our precise deviations results
by comparing the approximation of the tail probability $\mathbb{P}(N_{t}\geq xt)$
by using our formulas and by using Monte Carlo simulations.
Since the event $\{N_{t}\geq xt\}$ we are interested in is a rare event, 
we will first develop the importance sampling.
Rare event simulations using importance sampling
have been studied for affine point processes, a generalization
of the linear Hawkes process when the exciting function
is exponential, see \cite{Zhang,ZhangII}.
We are interested in the importance sampling 
for linear Hawkes processes with general exciting function,
which are non-Markovian in general, and is not covered
in \cite{Zhang,ZhangII}.
The importance sampling has also been used
to estimate the ruin probability in a risk model where
the arrival process of the claims follows a Hawkes process \cite{Stabile}.

We are interested to estimate right tail probability
$\mathbb{P}\left(N_{t}\geq xt\right)$, where $x>\frac{\nu}{1-\Vert h\Vert_{L^{1}}}$.
Note that $\{N_{t}\geq xt\}$ is a rare event
for $x>\frac{\nu}{1-\Vert h\Vert_{L^{1}}}$. 
The idea of the importance sampling is to change
the measure from $\mathbb{P}$ to a new measure $\hat{\mathbb{P}}$
under which the event $\{N_{t}\geq xt\}$ becomes a typical event.

Let us define a new probability measure $\hat{\mathbb{P}}$
under which the $N_{t}$ process is again a linear Hawkes process,
but with baseline intensity $\gamma\nu$ and the exciting 
function $\gamma h(\cdot)$, where $\gamma$ is a positive constant
to be chosen later. 
Under the new measure, the $N_{t}$ process
has the intensity $\hat{\lambda}_{t}=\gamma\lambda_{t}$.
We have the following result which will be proved in the Appendix.

\begin{proposition}\label{Prop:IS}
\begin{equation}
\mathbb{P}(N_{t}\geq xt)
=e^{-tI(x)}\hat{\mathbb{E}}\left[1_{\{N_{t}\geq xt\}}
\cdot 
e^{((\gamma-1)\Vert h\Vert_{L^{1}}-\log\gamma)(N_{t}-xt)-(\gamma-1)\int_{0}^{t}H(t-u)dN_{u}}\right],
\end{equation}
where $H(t):=\int_{t}^{\infty}h(s)ds$, 
and $I(x)$ is given in \eqref{Ifunction}.
\end{proposition}

Our numerical illustrations include three different methods: (1) importance sampling; (2) precise deviations
up to the first-order approximation; (3) precise deviations up to the second-order approximation.

(1) \textit{Importance sampling}. 
We simulate under $\hat{\mathbb{E}}$
a Hawkes process $N_{t}$ with intensity:
\begin{equation}
\hat{\lambda}_{t}=\gamma\nu+\gamma\int_{0}^{t-}h(t-s)dN_{s},
\end{equation}
where $\gamma=\frac{x}{\nu+\Vert h\Vert_{L^{1}}x}$.
Note that 
under the new measure $\hat{\mathbb{P}}$, we have
$\hat{\mathbb{P}}\left(\frac{N_{t}}{t}\rightarrow\frac{\gamma\nu}{1-\gamma\Vert h\Vert_{L^{1}}}\right)=1$.
By choosing $\gamma$ as $\gamma=\frac{x}{\nu+\Vert h\Vert_{L^{1}}x}$,
we have $\frac{\gamma\nu}{1-\gamma\Vert h\Vert_{L^{1}}}=x$,
and it follows that $\{N_{t}\geq xt\}$ is a typical event under the new measure $\hat{\mathbb{P}}$.

Using importance sampling Monte Carlo method (Proposition \ref{Prop:IS}), we estimate
\begin{equation}\label{form1}
e^{-tI(x)}\hat{\mathbb{E}}\left[1_{\{N_{t}\geq xt\}}
\cdot 
e^{((\gamma-1)\Vert h\Vert_{L^{1}}-\log\gamma)(N_{t}-xt)-(\gamma-1)\int_{0}^{t}H(t-u)dN_{u}}\right],
\end{equation}
where $H(t):=\int_{t}^{\infty}h(s)ds$ denotes
the right tail of the exciting function, and
$I(x)=x\log\left(\frac{x}{\nu+x\Vert h\Vert_{L^{1}}}\right)-x+x\Vert h\Vert_{L^{1}}+\nu$
as is given in \eqref{Ifunction}.

(2) \textit{Precise deviations up to the first-order approximation}.
We approximate $\mathbb{P}(N_{t}\geq xt)$ by the formula:
\begin{equation}\label{form2}
e^{-tI(x)}\sqrt{\frac{I''(x)}{2\pi t}}c_{0},
\end{equation}
where $c_{0}=\frac{1}{1-e^{-\theta^{\ast}}}
\psi(\theta^{\ast})$, and 
$I''(x)=\frac{\nu^{2}}{x(\nu+\Vert h\Vert_{L^{1}}x)^{2}}$
is the second derivative of $I(x)$ defined in \eqref{Ifunction},
and $\theta^{\ast}=\log\left(\frac{x}{\nu+\Vert h\Vert_{L^{1}}x}\right)
-\frac{\Vert h\Vert_{L^{1}}x}{\nu+\Vert h\Vert_{L^{1}}x}+\Vert h\Vert_{L^{1}}$,
where
\begin{equation}
\psi(\theta)=e^{\nu\int_{0}^{\infty}[F(s;\theta)-x(\theta)]ds},
\end{equation}
where $F$ is the unique solution that satisfies
\begin{equation}\label{num1}
F(t;\theta)=e^{\theta+\int_{0}^{t}(F(t-s;\theta)-1)h(s)ds},
\end{equation}
with the constraint $|F(t;\theta)|\leq\frac{1}{\Vert h\Vert_{L^{1}}}$,
and $x(\theta)=F(\infty;\theta)$ is the unique solution to:
\begin{equation}\label{num2}
x(\theta)=e^{\theta+(x(\theta)-1)\Vert h\Vert_{L^{1}}},
\end{equation}
with the constraint $|x(\theta)|\leq\frac{1}{\Vert h\Vert_{L^{1}}}$.
Note that we need to solve \eqref{num1} and \eqref{num2} numerically.

(3) \textit{Precise deviations up to the second-order approximation}.
We approximate $\mathbb{P}(N_{t}\geq xt)$ by adding a higher-order term
to Eqn. \eqref{form2} as follows:
\begin{equation}\label{form3}
e^{-tI(x)}\sqrt{\frac{I''(x)}{2\pi t}}\left(c_{0}+\frac{c_{1}}{t}\right),
\end{equation}
where $c_{0}=\frac{\psi(\theta^{\ast})}{1-e^{-\theta^{\ast}}}$, 
and $c_{1}=\frac{b_{1}}{1-e^{-\theta^{\ast}}}$, 
where the formula for $b_{1}$ can be computed by applying Proposition \ref{prop:ab}
and Proposition \ref{prop:etapsi}. The details for the computations of $b_{1}$ can be found in the Appendix.

In our numerical illustrations, we take baseline intensity $\nu=1$,
and consider two different exciting functions:
$h(t)=e^{-2t}$;
$h(t)=\frac{1}{(1+t)^{3}}$.
In both cases, $\Vert h\Vert_{L^{1}}=\frac{1}{2}$,
and in the first case, $H(t)=\int_{t}^{\infty}h(s)ds=\frac{1}{2}e^{-2t}$,
and in the second case, 
$H(t)=\int_{t}^{\infty}h(s)ds=\frac{1}{2(1+t)^{2}}$.
We then compare the three different methods (1), (2) and (3)
by comparing \eqref{form1}, \eqref{form2} and \eqref{form3}.
We summarize the results in Table \ref{Table:Exp} 
when the exciting function has exponential decay $h(t)=e^{-2t}$
and Table \ref{Table:Poly}
when the exciting function has polynomial decay $h(t)=\frac{1}{(1+t)^{3}}$.
In Table \ref{Table:Exp}, we take $x=4$ and $x=5$, and consider
the times $t=5,10,25,40,50$. 
Numerically, we compute that $c_{0}=4.8$ and
$c_{1}=-22$.  
In Table \ref{Table:Poly}, we take $x=4$ and $x=5$, and consider
the times $t=5,10,25,40,50$.
Numerically, we compute that $c_{0}=3.51$ and $c_{1}=-24$.
In both tables, the column IS provides the numerical results
using the importance sampling;
the column 1st Order provides the numerical results
using the precise deviations formula
up to the first-order approximation;
the column 2nd Order provides the numerical results
using the precise deviations formula
up to the second-order approximation.
In both tables, we observe that as time $t$ gets larger,
the approximations get better. 
First-order approximation tend to overestimate the tail probability
while second-order approximation tend to underestimate
the tail probability in Table \ref{Table:Exp} while
that is not the case in Table \ref{Table:Poly}. This is due
to the positivity of $c_{0}$ and negativity of $c_{1}$.
As a result, when $t$ is small, second-order approximation
could give negative values which are unrealistic, e.g. when $t=5$ in Table \ref{Table:Exp} and Table \ref{Table:Poly}. 
In both tables, as $t$ becomes larger, second-order approximation provides
better approximation than the first-order approximation.

\begin{table}[htb]
\centering 
\begin{tabular}{|c||c|c|c||c||c|c|c|} 
\hline 
$x=4$ & IS & 1st Order & 2nd Order & $x=5$ & IS & 1st Order & 2nd Order  
\\
\hline
\hline
$t=5$ & 4.71E-02 & 9.19E-02 & -2.08E-02 & $t=5$ & 1.37E-02 & 2.68E-02 & 1.40E-03
\\
\hline
$t=10$ & 2.06E-02 & 3.06E-02 & 1.18E-02 & $t=10$ & 3.19E-03 & 4.59E-03 & 2.41E-03
\\
\hline
$t=25$ & 1.62E-03 & 2.02E-03 & 1.52E-03 & $t=25$ & 3.49E-05 & 4.14E-05 & 3.35E-05
\\
\hline
$t=40$ & 1.46E-04 & 1.66E-04 & 1.41E-04 & $t=40$ & 4.20E-07 & 4.66E-07 & 4.11E-07
\\
\hline
$t=50$ & 2.93E-05 & 3.29E-05 & 2.89E-05 & $t=50$ & 2.14E-08 & 2.45E-08 & 2.22E-08
\\
\hline 
\end{tabular}
\caption{
Numerical illustration of $\mathbb{P}(N_{t}\geq xt)$ 
for $h(t)=e^{-2t}$, $\nu=1$ and $x=4,5$.}
\label{Table:Exp} 
\end{table}

\begin{table}[htb]
\centering 
\begin{tabular}{|c||c|c|c||c||c|c|c|} 
\hline 
$x=4$ & IS & 1st Order & 2nd Order & $x=5$ & IS & 1st Order & 2nd Order  
\\
\hline
\hline
$t=5$ & 3.04E-02 & 7.39E-02 & -4.37E-02 & $t=5$ & 7.50E-03 & 1.94E-02 & -7.60E-03
\\
\hline
$t=10$ & 1.33E-02 & 2.46E-02 & 5.02E-02 & $t=10$ & 1.63E-03 & 3.33E-03 & 1.01E-03
\\
\hline
$t=25$ & 1.12E-03 & 1.62E-03 & 1.11E-03 & $t=25$ & 1.95E-05 & 2.99E-05 & 2.16E-05
\\
\hline
$t=40$ & 1.03E-04 & 1.34E-04 & 1.07E-04 & $t=40$ & 2.39E-07 & 3.38E-07 & 2.79E-07
\\
\hline
$t=50$ & 2.28E-05 & 2.65E-05 & 2.23E-05 & $t=50$ & 1.28E-08 & 1.78E-08 & 1.53E-08
\\
\hline 
\end{tabular}
\caption{
Numerical illustration of $\mathbb{P}(N_{t}\geq xt)$ 
for $h(t)=1/(1+t)^{3}$, $\nu=1$ and $x=4,5$.}
\label{Table:Poly} 
\end{table}


\section{Appendix: Proofs}\label{ProofsSection}

\subsection{Proofs of the Results in Section \ref{LDPSection}}

\begin{proof}[Proof of Proposition \ref{xprop}]
Consider $\theta\in\mathbb{R}$, and 
$x(\theta)=e^{\theta+\Vert h\Vert_{L^{1}}(x(\theta)-1)}$.
Then $x(\theta)$ is increasing in $\theta$
by $x(\theta)=\mathbb{E}[e^{\theta S_{\infty}}]$ and the definition of $S_{\infty}$.
Moreover, for $\theta=\theta_{c}=\Vert h\Vert_{L^{1}}-1-\log\Vert h\Vert_{L^{1}}$, 
we have
$x(\theta_{c})=\frac{1}{\Vert h\Vert_{L^{1}}}e^{\Vert h\Vert_{L^{1}}-1+\Vert h\Vert_{L^{1}}(x(\theta_{c})-1)}$,
which implies that $x(\theta_{c})\Vert h\Vert_{L^{1}}=1$.
Thus, for any $\theta\leq\theta_{c}$, $x(\theta)\Vert h\Vert_{L^{1}}\leq 1$,
which proves (i).
We also notice that
$x'(\theta)=\left(1+\Vert h\Vert_{L^{1}}x'(\theta)\right)x(\theta)$,
and therefore as $\theta\uparrow\theta_{c}$,
$x'(\theta)=\frac{x(\theta)}{1-\Vert h\Vert_{L^{1}}x(\theta)}\rightarrow\infty$,
that proves (ii), which also follows from Theorem 3.2.1. in \cite{Bordenave}.
\end{proof}

\begin{proof}[Proof of Proposition \ref{uniqueProp}]
Suppose $F_{1}$ and $F_{2}$ are two solutions that
satisfy \eqref{FEqn} with $|F_{j}|\leq\frac{1}{\Vert h\Vert_{L^{1}}}$, for $j=1,2$.
Note that for any $z_{1},z_{2}\in\mathbb{C}$,
\begin{align*}
|e^{z_{1}}-e^{z_{2}}|
&=\left|\sum_{n=1}^{\infty}\frac{z_{1}^{n}-z_{2}^{n}}{n!}\right|
\\
&\leq|z_{1}-z_{2}|\sum_{n=1}^{\infty}\frac{|z_{1}|^{n-1}+|z_{1}|^{n-2}|z_{2}|+\cdots+|z_{1}||z_{2}|^{n-2}+|z_{2}|^{n-1}}{n!}
\\
&\leq|z_{1}-z_{2}|\sum_{n=1}^{\infty}\frac{n(|z_{1}|+|z_{2}|)^{n-1}}{n!}=|z_{1}-z_{2}|e^{|z_{1}|+|z_{2}|}.
\end{align*}
Hence, for any $T>0$ and for any $0\leq t\leq T$,
\begin{align*}
|F_{1}(t;z)-F_{2}(t;z)|
&=|e^{z}|\left|e^{\int_{0}^{t}(F_{1}(t-s;z)-1)h(s)ds}-e^{\int_{0}^{t}(F_{2}(t-s;z)-1)h(s)ds}\right|
\\
&\leq |e^{z}|\left|\int_{0}^{t}(F_{1}(t-s;z)-F_{2}(t-s;z))h(s)ds\right|
\\
&\qquad\qquad\qquad\qquad
\cdot
e^{|\int_{0}^{t}(F_{1}(t-s;z)-1)h(s)ds|+|\int_{0}^{t}(F_{2}(t-s;z)-1)h(s)ds|}
\\
&\leq |e^{z}|\Vert h\Vert_{L^{\infty}[0,T]}
e^{2\left(\frac{1}{\Vert h\Vert_{L^{1}}}+1\right)\Vert h\Vert_{L^{1}}}
\int_{0}^{t}|F_{1}(s;z)-F_{2}(s;z)|ds.
\end{align*}
Note that $F_{1}(0;z)=F_{2}(0;z)=e^{z}$ from \eqref{FEqn}.
By Gronwall's inequality, we conclude that $F_{1}\equiv F_{2}$.
\end{proof}


\begin{proof}[Proof of Lemma \ref{divLemma}]
Note that $Y=\sum_{i=1}^{K}Z_{i}$ in distribution,
where $K$ is Poisson random variable with parameter $\nu$
and $Z_{i}$ are i.i.d. random variable interpreted
as the total number of the nodes in a Galton-Watson tree
with the number of children being born in each generation
Poisson distributed with parameter $\Vert h\Vert_{L^{1}}$, independent of $K$.
A compound Poisson random variable is infinitely divisible, which
is evident from the expression of its characteristic function.
Hence, $Y$ has an infinitely divisible distribution.
\end{proof}

\begin{proof}[Proof of Lemma \ref{MainLemma}]
Firstly, it is obvious that for any $s>0$, and $t>0$,
$F(s,z)$, $x(z)$, and $\int_0^t (F(s,z)-x(z)) ds$
are analytic in $\mathcal{S}_{(-\infty,\theta_{c})}$.
Since $\Vert h\Vert_{L^{1}}<1$ and $x(0)=1$, 
for $\mathcal{R}(z)<\Vert h\Vert_{L^{1}}-1-\log\Vert h\Vert_{L^{1}}$, 
we have 
$
|x(z)|\Vert h\Vert_{L^{1}}
\leq x(\mathcal{R}(z))\Vert h\Vert_{L^{1}}
<1.
$
Therefore, for any compact subset $K\subset \{z\in\mathbb C;\mathcal{R}(z)<\Vert h\Vert_{L^{1}}-1-\log\Vert h\Vert_{L^{1}}\}$,
we have
\begin{equation*}
\sup_{z\in K}|x(z)|\Vert h\Vert_{L^{1}}
\leq x(\mathcal{R}(z))\Vert h\Vert_{L^{1}}
<1.
\end{equation*}
Since $\int_{s}^{\infty}h(u)du\rightarrow 0$ as $s\rightarrow\infty$ and $F(s;z)=\mathbb{E}[e^{zS_{s}}]\rightarrow x(z)=\mathbb{E}[e^{zS_{\infty}}]$,
we have that
\begin{equation*}
\int_{0}^{s}\sup_{z\in K}[F(s-u;z)-x(z)]h(u)du\rightarrow 0,
\qquad
\text{as $s\rightarrow\infty$}.
\end{equation*}
Note that
\begin{equation*}
F(s;z)-x(z)=x(z)\left[e^{\int_{0}^{s}[F(s-u;z)-x(z)]h(u)du-(x(z)-1)\int_{s}^{\infty}h(u)du}-1\right],
\end{equation*}
and for any fixed $\delta>0$
such that $(1+\delta)\sup_{z\in K}|x(z)|\Vert h\Vert_{L^{1}}<1$, 
there exists $M>0$, so that for any $s\geq M$ and $z\in K$, we have
\begin{equation*}
|F(s;z)-x(z)|\leq (1+\delta)|x(z)| \left[\int_{0}^{s}|F(s-u;z)-x(z)|h(u)du+|x(z)-1|\int_{s}^{\infty}h(u)du\right].
\end{equation*}
Therefore, we get that for any $T> M$,
\begin{align*}
\int_{M}^{T}\sup_{z\in K}|F(s;z)-x(z)|ds
&\leq
(1+\delta)\sup_{z\in K}|x(z)|\int_{M}^{T}\int_{0}^{s}\sup_{z\in K}|F(s-u;z)-x(z)|h(u)duds
 \\
 &\qquad\qquad
 +(1+\delta)\sup_{z\in K}|x(z)||x(z)-1|\int_{M}^{T}\int_{s}^{\infty}h(u)duds
\\
&\leq
(1+\delta)\sup_{z\in K}|x(z)|\int_{0}^{T}\int_{0}^{s}\sup_{z\in K}|F(s-u;z)-x(z)|h(u)duds
 \\
 &\qquad\qquad
+(1+\delta)\sup_{z\in K}|x(z)||x(z)-1|\int_{0}^{\infty}\int_{s}^{\infty}h(u)duds,
\end{align*}
which implies that 
\begin{align*}
&\int_{0}^{T}\sup_{z\in K}|F(s;z)-x(z)|ds
\\
 &\leq \int_{0}^{M}\sup_{z\in K}|F(s;z)-x(z)|ds
+(1+\delta)\sup_{z\in K}|x(z)|\int_{0}^{T}\int_{0}^{s}\sup_{z\in K}|F(s-u;z)-x(z)|h(u)duds
 \\
 &\qquad\qquad\qquad\qquad
+(1+\delta)\sup_{z\in K}|x(z)||x(z)-1|\int_{0}^{\infty}\int_{s}^{\infty}h(u)duds
\\
 &=\int_{0}^{M}\sup_{z\in K}|F(s;z)-x(z)|ds
+(1+\delta)\sup_{z\in K}|x(z)|\int_{0}^{T}\left[\int_{u}^{T}\sup_{z\in K}|F(s-u;z)-x(z)|ds\right]h(u)du
 \\
 &\qquad\qquad\qquad\qquad
+(1+\delta)\sup_{z\in K}|x(z)||x(z)-1|\int_{0}^{\infty}\int_{s}^{\infty}h(u)duds
\\
&\leq \int_{0}^{M}\sup_{z\in K}|F(s;z)-x(z)|ds
+
(1+\delta)\sup_{z\in K}|x(z)| \Vert h\Vert_{L^{1}}\int_{0}^{T}|F(s;z)-x(z)|ds
\\
&\qquad\qquad\qquad\qquad
+(1+\delta)\sup_{z\in K}|x(z)|\left[\sup_{z\in K}|x(z)|+1\right]\int_{0}^{\infty}sh(s)ds,
\end{align*}
which holds for any $T>M$, and thus we have  
\begin{equation}
\begin{aligned}
&\int_{0}^{\infty}\sup_{z\in K}|F(s;z)-x(z)|ds\\
&\leq\frac{ \int_{0}^{M}\sup_{z\in K}|F(s;z)-x(z)|ds
+(1+\delta)\sup_{z\in K}|x(z)|\left[\sup_{z\in K}|x(z)|+1\right]\int_{0}^{\infty}sh(s)ds}
{1-(1+\delta)\sup_{z\in K}|x(z)| \Vert h\Vert_{L^{1}}}.
\end{aligned}
\end{equation}
Hence, we conclude that $\int_{t}^{\infty}\sup_{z\in K}|F(s;z)-x(z)|ds\to 0$ as $t\to \infty$, and so
\begin{equation}\label{F:t:infty}
\int_0^t (F(s,z)-x(z)) ds\to \int_0^\infty (F(s,z)-x(z)) ds,
\end{equation}
as $t\rightarrow\infty$, locally uniformly in $z$ for $\mathcal{R}(z)<\theta_{c}$.
Hence, $\varphi(z)$ is well-defined and is analytic in $\mathcal{S}_{(-\infty,\theta_{c})}$.
By equations \eqref{NEqn}, \eqref{F:t:infty} and the definitions 
of $\psi(z)$ and $\varphi(z)$ in \eqref{defn:psi:varphi}
we have proved that locally uniformly in $z$ for $\mathcal{R}(z)<\theta_{c}$,
\begin{equation*} 
e^{-t(\nu(x(z)-1))}\mathbb{E}[e^{zN_{t}}]=e^{\nu\int_0^t (F(s,z)-x(z)) ds}\to e^{\nu \varphi(z)}=\psi(z),
\qquad
\text{as $t\rightarrow\infty$}.
\end{equation*}

To show the mod-$\phi$ convergence at speed $O(t^{-v})$, that is,
for any compact set $K$, there exists some $C_{K}>0$ such that
$\sup_{z\in K}|e^{-t(\nu(x(z)-1))}\mathbb{E}[e^{zN_{t}}]
-e^{\nu \varphi(z)}|\leq C_{K}t^{-v}$, it suffices to show that
for any compact set $K\subset \{z\in\mathbb C;\mathcal{R}(z)<\Vert h\Vert_{L^{1}}-1-\log\Vert h\Vert_{L^{1}}\}$,
\begin{equation}\label{t:v:integral}
\int_{0}^{\infty}t^{v}\sup_{z\in K}|F(t;z)-x(z)|dt<\infty.
\end{equation}
To see this, notice that 
\begin{align*}
\sup_{z\in K}\left|e^{-t(\nu(x(z)-1))}\mathbb{E}[e^{zN_{t}}]
-e^{\nu \varphi(z)}\right|
&=\sup_{z\in K}\left|e^{\nu\varphi(z)}
\left(e^{\nu\int_{t}^{\infty}(F(s;z)-x(z))ds}-1\right)\right|
\\
&\leq
\sup_{z\in K}e^{\nu|\varphi(z)|}
\left(e^{\nu\int_{t}^{\infty}\sup_{z\in K}|F(s;z)-x(z)|ds}-1\right).
\end{align*}
Thus it suffices to show that 
\begin{equation}\label{tail:t:v}
\int_{t}^{\infty}\sup_{z\in K}|F(s;z)-x(z)|ds\leq c_{K}t^{-v},
\end{equation}
for some $c_{K}>0$.
For any $T>0$, by integration by parts,
\begin{align}
&\int_{0}^{T}t^{v}\sup_{z\in K}|F(t;z)-x(z)|dt
\nonumber
\\
&=-\int_{0}^{T}t^{v}d\left(\int_{t}^{\infty}\sup_{z\in K}|F(s;z)-x(z)|ds\right)
\nonumber
\\
&=-T^{v}\int_{T}^{\infty}\sup_{z\in K}|F(s;z)-x(z)|ds
+v\int_{0}^{T}t^{v-1}\int_{t}^{\infty}\sup_{z\in K}|F(s;z)-x(z)|dsdt.\label{integration:by:parts}
\end{align}
On the other hand, by Fubini's theorem and \eqref{t:v:integral}, we have
\begin{align}
v\int_{0}^{\infty}t^{v-1}\int_{t}^{\infty}\sup_{z\in K}|F(s;z)-x(z)|dsdt
&=v\int_{0}^{\infty}\int_{0}^{s}t^{v-1}\sup_{z\in K}|F(s;z)-x(z)|dtds
\nonumber
\\
&=\int_{0}^{\infty}s^{v}\sup_{z\in K}|F(s;z)-x(z)|dtds<\infty.\label{integral:equiv}
\end{align}
Hence we get
$\lim_{T\rightarrow\infty}T^{v}\int_{T}^{\infty}\sup_{z\in K}|F(s;z)-x(z)|ds=0$
by letting $T\rightarrow\infty$ in \eqref{integration:by:parts} and applying 
\eqref{integral:equiv}.
This implies \eqref{tail:t:v}.

Next, let us prove \eqref{t:v:integral}.
Notice that there exists $M>0$ so that for any $s\geq M$,
\begin{align*}
\sup_{z\in\mathbb{Z}}|F(s;z)-x(z)|
&\leq (1+\delta)\sup_{z\in K}|x(z)| 
\int_{0}^{s}\sup_{z\in K}|F(s-u;z)-x(z)|h(u)du
\\
&\qquad
+(1+\delta)\sup_{z\in K}|x(z)|\sup_{z\in K}|x(z)-1|\int_{s}^{\infty}h(u)du.
\end{align*}
Therefore, for every $s\geq 0$,
\begin{equation}\label{ineq:iterate}
\sup_{z\in K}|F(s;z)-x(z)|
\leq (1+\delta)\sup_{z\in K}|x(z)| 
\int_{0}^{s}\sup_{z\in K}|F(s-u;z)-x(z)|h(u)du
+g(s),
\end{equation}
where
\begin{equation}
g(s):=C_{1}\int_{s}^{\infty}h(u)du
+C_{2}1_{\{s\leq M\}},
\end{equation}
where
\begin{equation}
C_{1}:=(1+\delta)\sup_{z\in K}|x(z)|\sup_{z\in K}|x(z)-1|,
\qquad
C_{2}:=\sup_{0\leq s\leq M}\sup_{z\in K}|F(s;z)-x(z)|.
\end{equation}
Let $p(s):=\sup_{z\in K}|F(s;z)-x(z)|$
and $q(s):=(1+\delta)\sup_{z\in K}|x(z)|h(s)$
for every $s\geq 0$. 
Then, \eqref{ineq:iterate} can be re-written as
$p(s)\leq\int_{0}^{s}q(s-u)p(u)du+g(s)$, $s\geq 0$,
which implies
\begin{align*}
p(s)&\leq\int_{0}^{s}q(s-u)\left(\int_{0}^{u}q(u-v)p(v)dv+g(u)\right)du+g(s)
\\
&=\int_{0}^{s}q^{\ast 2}(s-u)p(u)du+\int_{0}^{s}q^{\ast 1}(s-u)g(u)du+g(s),
\qquad s\geq 0.
\end{align*}
By iterating and following the same argument as in the proof
of generalized Gronwall's inequality in \cite{CM},
we have
\begin{equation*}
p(s)\leq g(s)+\int_{0}^{s}\sum_{k=1}^{\infty}q^{\ast k}(s-u)g(u)du,
\end{equation*}
which is equivalent to
\begin{equation}
\sup_{z\in K}|F(s;z)-x(z)|
\leq g(s)+\int_{0}^{s}G(s-u)g(u)du,
\end{equation}
where
\begin{equation}
G(t):=\sum_{k=1}^{\infty}\left((1+\delta)\sup_{z\in K}|x(z)| \right)^{k}h^{\ast k}(t),
\end{equation}
where $h^{\ast k}(t)$ is the $k$-th order convolution of $h$, i.e.
$h^{\ast 0}(t):=h(t)$ and $h^{\ast m}(t):=\int_{0}^{t}h^{\ast(m-1)}(t-s)h(s)ds$
for every $m\in\mathbb{N}$.
Hence, it remains to show that
\begin{align}
&\int_{0}^{\infty}s^{v}g(s)ds<\infty,\label{remain:1}
\\
&\int_{0}^{\infty}s^{v}\int_{0}^{s}G(s-u)g(u)duds<\infty.\label{remain:2}
\end{align}
Let us first prove \eqref{remain:1}.
Note that
\begin{equation}
\int_{0}^{\infty}s^{v}g(s)ds
=C_{1}\int_{0}^{\infty}s^{v}\int_{s}^{\infty}h(u)duds
+C_{2}\int_{0}^{M}s^{v}dv,
\end{equation}
and by our assumption,
\begin{equation}
\int_{0}^{\infty}s^{v}\int_{s}^{\infty}h(u)duds
=\int_{0}^{\infty}\int_{0}^{u}s^{v}h(u)dsdu
=\frac{1}{v+1}\int_{0}^{\infty}u^{v+1}h(u)du<\infty,
\end{equation}
thus \eqref{remain:1} follows.

Next, let us prove \eqref{remain:2}.
Note that
\begin{align*}
&\int_{0}^{\infty}s^{v}\int_{0}^{s}G(s-u)g(u)duds
\\
&=C_{1}\int_{0}^{\infty}s^{v}\int_{0}^{s}G(s-u)\int_{u}^{\infty}h(x)dxduds
+C_{2}\int_{0}^{\infty}s^{v}\int_{0}^{s}G(s-u)1_{\{u\leq M\}}duds,
\end{align*}
and it is easy to check that
\begin{align*}
\int_{0}^{\infty}s^{v}\int_{0}^{s}G(s-u)1_{\{u\leq M\}}duds
&=\int_{0}^{\infty}\int_{u}^{\infty}s^{v}G(s-u)ds1_{\{u\leq M\}}du
\\
&=\int_{0}^{M}\int_{0}^{\infty}(s+u)^{v}G(s)dsdu
\\
&\leq 2^{v-1}M\int_{0}^{\infty}(s^{v}+M^{v})G(s)ds,
\end{align*}
where we used the inequality $(a+b)^{v}\leq 2^{v-1}(a^{v}+b^{v})$ for any $a,b>0$ and $v\geq 1$.
We can compute that
\begin{equation*}
\int_{0}^{\infty}G(s)ds
=\sum_{k=1}^{\infty}\left((1+\delta)\sup_{z\in K}|x(z)|\right)^{k}
\int_{0}^{\infty}h^{\ast k}(s)ds
=\sum_{k=1}^{\infty}\left((1+\delta)\sup_{z\in K}|x(z)|\right)^{k}
\Vert h\Vert_{L^{1}}^{k},
\end{equation*}
which is finite since $(1+\delta)\sup_{z\in K}|x(z)|\Vert h\Vert_{L^{1}}<1$.
Next, let us show that 
\begin{equation}
\int_{0}^{\infty}s^{v}G(s)ds<\infty.
\end{equation}
Notice that
\begin{align*}
\int_{0}^{\infty}s^{v}G(s)ds
&=\sum_{k=1}^{\infty}\left((1+\delta)\sup_{z\in K}|x(z)|\right)^{k}
\int_{0}^{\infty}s^{v}h^{\ast k}(s)ds
\\
&=\sum_{k=1}^{\infty}\left((1+\delta)\sup_{z\in K}|x(z)|\right)^{k}
\int_{0}^{\infty}s^{v}\int_{0}^{s}h(s-u)h^{\ast(k-1)}(u)duds
\\
&=\sum_{k=1}^{\infty}\left((1+\delta)\sup_{z\in K}|x(z)|\right)^{k}
\int_{0}^{\infty}\int_{u}^{\infty}s^{v}h(s-u)dsh^{\ast(k-1)}(u)du
\\
&=\sum_{k=1}^{\infty}\left((1+\delta)\sup_{z\in K}|x(z)|\right)^{k}\int_{0}^{\infty}\int_{0}^{\infty}(s+u)^{v}h(s)dsh^{\ast(k-1)}(u)du.
\end{align*}
Note that for any $\delta'>0$, there exists
some $C(\delta')>0$ such that
for any $s,u\geq 0$, 
\begin{equation}
(s+u)^{v}\leq C(\delta')s^{v}+(1+\delta')u^{v}.
\end{equation}
Therefore
\begin{align*}
&\int_{0}^{\infty}s^{v}h^{\ast k}(s)ds
\\
&\leq
C(\delta')\int_{0}^{\infty}s^{v}h(s)ds
\int_{0}^{\infty}h^{\ast(k-1)}(u)du
+(1+\delta')\int_{0}^{\infty}h(s)ds\int_{0}^{\infty}u^{v}h^{\ast(k-1)}(u)du
\\
&=C(\delta')\int_{0}^{\infty}s^{v}h(s)ds\Vert h\Vert_{L^{1}}^{k-1}
+(1+\delta')\Vert h\Vert_{L^{1}}\int_{0}^{\infty}u^{v}h^{\ast(k-1)}(u)du.
\end{align*}
Let us define $A_{k}:=\int_{0}^{\infty}s^{v}h^{\ast k}(s)ds$.
Then, we have:
\begin{equation}
A_{k}\leq C(\delta')\Vert h\Vert_{L^{1}}^{k-1}A_{1}
+(1+\delta')\Vert h\Vert_{L^{1}}A_{k-1}.
\end{equation}
It follows that
\begin{align*}
A_{k}&\leq
C(\delta')\Vert h\Vert_{L^{1}}^{k-1}
\left[1+(1+\delta')\Vert h\Vert_{L^{1}}
+((1+\delta')\Vert h\Vert_{L^{1}})^{2}
+\cdots
+((1+\delta')\Vert h\Vert_{L^{1}})^{k-2}\right]A_{1}
\\
&\qquad\qquad\qquad
+\left((1+\delta')\Vert h\Vert_{L^{1}}\right)^{k-1}A_{1}.
\end{align*}
Choose $\delta'>0$ to be sufficiently small so that $(1+\delta')\Vert h\Vert_{L^{1}}<1$.
Then, we have
\begin{align*}
A_{k}&\leq\frac{C(\delta')}{1-(1+\delta')\Vert h\Vert_{L^{1}}}\Vert h\Vert_{L^{1}}^{k-1}A_{1}
+(1+\delta')^{k-1}\Vert h\Vert_{L^{1}}^{k-1}A_{1}
\\
&\leq
\left(\frac{C(\delta')}{1-(1+\delta')\Vert h\Vert_{L^{1}}}+1\right)(1+\delta')^{k-1}\Vert h\Vert_{L^{1}}^{k-1}A_{1}
\end{align*}
Choose $\delta'>0$ to be sufficiently small
so that $(1+\delta)(1+\delta')\sup_{z\in K}|x(z)|\Vert h\Vert_{L^{1}}<1$.
Hence, we conclude that
\begin{align*}
&\int_{0}^{\infty}s^{v}G(s)ds
\\
&=\sum_{k=1}^{\infty}\left((1+\delta)\sup_{z\in K}|x(z)|\right)^{k}A_{k}
\\
&\leq
(1+\delta)\sup_{z\in K}|x(z)|A_{1}
\left(\frac{C(\delta')}{1-(1+\delta')\Vert h\Vert_{L^{1}}}+1\right)
\sum_{k=1}^{\infty}\left((1+\delta)\sup_{z\in K}|x(z)|(1+\delta')\Vert h\Vert_{L^{1}}\right)^{k-1}
\\
&=
\frac{(1+\delta)\sup_{z\in K}|x(z)|A_{1}
\left(\frac{C(\delta')}{1-(1+\delta')\Vert h\Vert_{L^{1}}}+1\right)}{1-(1+\delta)(1+\delta')\sup_{z\in K}|x(z)|\Vert h\Vert_{L^{1}}}<\infty.
\end{align*}

Finally, we can compute that 
\begin{align*}
&\int_{0}^{\infty}s^{v}\int_{0}^{s}G(s-u)\int_{u}^{\infty}h(x)dxduds
\\
&=\int_{0}^{\infty}\int_{u}^{\infty}s^{v}G(s-u)ds\int_{u}^{\infty}h(x)dxdu
\\
&=\int_{0}^{\infty}\int_{0}^{\infty}(s+u)^{v}G(s)ds\int_{u}^{\infty}h(x)dxdu
\\
&\leq
2^{v-1}\int_{0}^{\infty}\int_{0}^{\infty}(s^{v}+u^{v})G(s)ds\int_{u}^{\infty}h(x)dxdu
\\
&=2^{v-1}\int_{0}^{\infty}s^{v}G(s)ds
\int_{0}^{\infty}sh(s)ds
+2^{v-1}\int_{0}^{\infty}G(s)ds
\cdot\frac{1}{v+1}\int_{0}^{\infty}s^{v+1}h(s)ds<\infty.
\end{align*}
This completes the proof.
\end{proof}


\begin{proof}[Proof of Theorem \ref{LDPThm}]
By Lemma \ref{divLemma}, and Lemma \ref{MainLemma}, 
we have established the mod-$\phi$ convergence.
Hence, by Theorem 3.2.2. in \cite{modphi}, 
for any $x>0$, and $tx\in\mathbb{N}$, 
\begin{equation}
\mathbb{P}(N_{t}=tx)
=\frac{e^{-tI(x)}}{\sqrt{2\pi t\eta''(\theta^{\ast})}}
\left(\psi(\theta^{\ast})+\frac{a_{1}}{t}+\frac{a_{2}}{t^{2}}
+\cdots+\frac{a_{v-1}}{t^{v-1}}
+O\left(\frac{1}{t^{v}}\right)\right),
\end{equation}
and for any $x>\eta'(0)=\frac{\nu}{1-\Vert h\Vert_{L^{1}}}$ and $tx\in\mathbb{N}$, 
\begin{equation}
\mathbb{P}(N_{t}\geq tx)
=\frac{e^{-tI(x)}}{\sqrt{2\pi t\eta''(\theta^{\ast})}}\frac{1}{1-e^{-\theta^{\ast}}}
\left(\psi(\theta^{\ast})+\frac{b_{1}}{t}+\frac{b_{2}}{t^{2}}
+\cdots+\frac{b_{v-1}}{t^{v-1}}
+O\left(\frac{1}{t^{v}}\right)\right),
\end{equation}
where $I(x)$ is defined in \eqref{Ifunction} and $\eta'(\theta^{\ast})=x$.
Note that $\eta(\theta)=\nu(x(\theta)-1)$, 
where $x(\theta)=e^{\theta+\Vert h\Vert_{L^{1}}(x(\theta)-1)}$.
Thus,
$\eta'(\theta)=\nu x'(\theta)$,
and
\begin{equation}
x'(\theta)=(1+\Vert h\Vert_{L^{1}}x'(\theta))e^{\theta+\Vert h\Vert_{L^{1}}(x(\theta)-1)}
=(1+\Vert h\Vert_{L^{1}}x'(\theta))x(\theta),
\end{equation}
which implies that 
\begin{equation}
x(\theta)=\frac{x'(\theta)}{1+\Vert h\Vert_{L^{1}}x'(\theta)}.
\end{equation}
Notice that $\eta'(\theta^{\ast})=x$, and thus
$x'(\theta^{\ast})=\frac{\eta'(\theta^{\ast})}{\nu}=\frac{x}{\nu}$,
and
$x(\theta^{\ast})=\frac{x'(\theta^{\ast})}{1+\Vert h\Vert_{L^{1}}x'(\theta^{\ast})}
\frac{x}{\nu+\Vert h\Vert_{L^{1}}x}$.
Therefore,
\begin{equation*}
\frac{x}{\nu}
=x'(\theta^{\ast})
=(1+\Vert h\Vert_{L^{1}}x'(\theta^{\ast}))e^{\theta^{\ast}+\Vert h\Vert_{L^{1}}(x(\theta^{\ast})-1)}
=\left(1+\frac{\Vert h\Vert_{L^{1}}x}{\nu}\right)
e^{\theta^{\ast}+\frac{\Vert h\Vert_{L^{1}}x}{\nu+\Vert h\Vert_{L^{1}}x}-\Vert h\Vert_{L^{1}}},
\end{equation*}
which implies that
$\theta^{\ast}=\log\left(\frac{x}{\nu+\Vert h\Vert_{L^{1}}x}\right)
-\frac{\Vert h\Vert_{L^{1}}x}{\nu+\Vert h\Vert_{L^{1}}x}+\Vert h\Vert_{L^{1}}$.
Hence, we have verified that
\begin{equation}
I(x)=\theta^{\ast}x-\eta(\theta^{\ast})
=x\log\left(\frac{x}{\nu+x\Vert h\Vert_{L^{1}}}\right)-x+x\Vert h\Vert_{L^{1}}+\nu.
\end{equation}

Moreover, by the property of Legendre transform,
$\eta''(\theta^{\ast})=\frac{1}{I''(x)}$,
which together with
\begin{equation}
I'(x)=\theta^{\ast}=\log\left(\frac{x}{\nu+\Vert h\Vert_{L^{1}}x}\right)
-\frac{\Vert h\Vert_{L^{1}}x}{\nu+\Vert h\Vert_{L^{1}}x}+\Vert h\Vert_{L^{1}},
\end{equation}
implies that
\begin{equation}\label{Itwice}
I''(x)=\frac{\nu^{2}}{x(\nu+\Vert h\Vert_{L^{1}}x)^{2}}.
\end{equation}
\end{proof}


\subsection{Proofs of the Results in Section \ref{MDPSection}}

\begin{proof}[Proof of Theorem \ref{MDPThm}]
Since we have established mod-$\phi$ convergence,
and $N_{t}$ is lattice distributed, 
the result follows from Corollary 3.3.5. in \cite{modphi},
which is restated in our \eqref{precise:3}.

Next, let us show that \eqref{Ii} holds. Recall from \eqref{Itwice}
that $I''(x)=\frac{\nu^{2}}{x(\nu+\Vert h\Vert_{L^{1}}x)^{2}}$.
Therefore, for any $i\geq 2$, by Leibniz formula, 
\begin{align*}
I^{(i)}(x)
&=\nu^{2}\frac{d^{i-2}}{dx^{i-2}}x^{-1}\cdot(\nu+\Vert h\Vert_{L^{1}}x)^{-2}
\\
&=\nu^{2}\sum_{j=0}^{i-2}\frac{(i-2)!}{(i-2-j)!j!}(-1)^{i-2}(i-2-j)!
\Vert h\Vert_{L^{1}}^{j}(j+1)!x^{-1-(i-2-j)}(\nu+\Vert h\Vert_{L^{1}}x)^{-j-2}
\\
&=\nu^{2}(i-2)!(-1)^{i-2}\frac{x^{1-i}}{(\nu+\Vert h\Vert_{L^{1}}x)^{2}}
\sum_{j=0}^{i-2}(j+1)\left(\frac{\Vert h\Vert_{L^{1}}x}{\nu+\Vert h\Vert_{L^{1}}x}\right)^{j}
\\
&=\nu^{2}(i-2)!(-1)^{i-2}\frac{x^{1-i}}{(\nu+\Vert h\Vert_{L^{1}}x)^{2}}
\frac{d}{dy}\sum_{j=0}^{i-2}y^{j+1}\bigg|_{y=\frac{\Vert h\Vert_{L^{1}}x}{\nu+\Vert h\Vert_{L^{1}}x}}
\\
&=\nu^{2}(i-2)!(-1)^{i-2}\frac{x^{1-i}}{(\nu+\Vert h\Vert_{L^{1}}x)^{2}}
\frac{(i-1)y^{i}-iy^{i-1}+1}{(y-1)^{2}}\bigg|_{y=\frac{\Vert h\Vert_{L^{1}}x}{\nu+\Vert h\Vert_{L^{1}}x}},
\end{align*}
which implies \eqref{Ii}.

Finally, let us compute $\eta'(0)$, $\eta''(0)$ for Theorem \ref{MDPThm}
as well as $\eta'''(0)$ for Remark \ref{m:4:remark}.
Let us recall that $\eta(\theta)=\nu(x(\theta)-1)$,
where $x(\theta)=e^{\theta+\Vert h\Vert_{L^{1}}(x(\theta)-1)}$.
Thus, we can compute $x'(\theta)=(1+\Vert h\Vert_{L^{1}}x'(\theta))x(\theta)$
and $x''(\theta)=\left[\Vert h\Vert_{L^{1}}x''(\theta)+(1+\Vert h\Vert_{L^{1}}x'(\theta))^{2}\right]x(\theta)$ and
\begin{align*}
&x'''(\theta)=\left[\Vert h\Vert_{L^{1}}x'''(\theta)
+2(1+\Vert h\Vert_{L^{1}}x'(\theta))\Vert h\Vert_{L^{1}}x''(\theta)\right]
x(\theta)
\\
&\qquad\qquad\qquad
+\left[\Vert h\Vert_{L^{1}}x''(\theta)+(1+\Vert h\Vert_{L^{1}}x'(\theta))^{2}\right]
(1+\Vert h\Vert_{L^{1}}x'(\theta))x(\theta).
\end{align*}
Note that $x(0)=1$. 
By letting $\theta=0$, we get the formulas for $\eta'(0)$, $\eta''(0)$ and $\eta'''(0)$.
\end{proof}


\subsection{Proof of Proposition \ref{prop:ab} (i): Derivations of $(a_{k})_{k=1}^{\infty}$}

Before we proceed, let us first introduce the Fa\`{a} di Bruno's formula that will
be used repeatedly in our proofs.

\begin{lemma}[Fa\`{a} di Bruno's formula]
\begin{equation}
\frac{d^{n}}{dx^{n}}f(g(x))
=\sum_{\mathcal{S}_{n}}\frac{n!f^{(m_{1}+\cdots+m_{n})}(g(x))}
{m_{1}!1!^{m_{1}}m_{2}!2!^{m_{2}}\cdots m_{n}!n!^{m_{n}}}
\cdot
\prod_{j=1}^{n}(g^{(j)}(x))^{m_{j}},
\end{equation}
where the sum is over the set $\mathcal{S}_{n}$ consisting of 
all the $n$-tuples of non-negative integers $(m_{1},\ldots,m_{n})$
satisfying the constraint
$1\cdot m_{1}+2\cdot m_{2}+3\cdot m_{3}+\cdots+n\cdot m_{n}=n$.
\end{lemma}

We recall from the proof of Theorem~3.2.2. and Remark~3.2.5. in \cite{modphi} that 
for $\mathbb{Z}$-valued random variables $X_{n}$, 
\begin{equation}
\mathbb{P}(X_{n}=t_{n}x)=\frac{e^{-t_{n}F(x)}}{\sqrt{2\pi t_{n}\eta''(\theta^{\ast})}}
\left(\psi(\theta^{\ast})+\frac{a_{1}}{t_{n}}+\cdots+\frac{a_{v-1}}{t_{n}^{v-1}}+O(t_{n}^{-v})\right),
\end{equation}
as $n\rightarrow\infty$, where
$a_{k}=\int_{-\infty}^{\infty}\alpha_{2k}(w)\frac{e^{-\frac{w^{2}}{2}}}{\sqrt{2\pi}}dw$,
where $\alpha_{k}$'s are defined via the expansion:
\begin{equation}
\sum_{k=0}^{2v-1}\frac{\alpha_{k}}{t_{n}^{k/2}}
+O(t_{n}^{-v})
=\sum_{k=0}^{2v-1}\frac{\psi^{(k)}(\theta^{\ast})}{k!}
\left(\frac{iw}{\sqrt{t_{n}\eta''(\theta^{\ast})}}\right)^{k}
e^{-\frac{w^{2}}{\eta''(\theta^{\ast})}\sum_{k=1}^{2v-1}\frac{\eta^{(k+2)}(\theta^{\ast})}{(k+2)!}\left(\frac{iw}{\sqrt{t_{n}\eta''(\theta^{\ast})}}\right)^{k}}.
\end{equation}

Let us define $f(x)=e^{-x}$ and
\begin{equation}
g(x)=\frac{w^{2}}{\eta''(\theta^{\ast})}\sum_{k=1}^{\infty}\frac{\eta^{(k+2)}(\theta^{\ast})}{(k+2)!}x^{k}.
\end{equation}
Then we can compute that  $g(0)=0$ and
for every $j\in\mathbb{N}$,
\begin{equation}
g^{(j)}(0)=\frac{w^{2}}{\eta''(\theta^{\ast})}\frac{\eta^{(j+2)}(\theta^{\ast})}{(j+2)(j+1)}.
\end{equation}
By Fa\`{a} di Bruno's formula, it follows that
\begin{equation}
f(g(x))=\sum_{n=0}^{\infty}
\sum\frac{n!(-1)^{m_{1}+\cdots+m_{n}}}{m_{1}!1!^{m_{1}}m_{2}!2!^{m_{2}}\cdots m_{n}!n!^{m_{n}}}
\cdot
\prod_{j=1}^{n}\left(\frac{w^{2}}{\eta''(\theta^{\ast})}\frac{\eta^{(j+2)}(\theta^{\ast})}{(j+2)(j+1)}\right)^{m_{j}}
\frac{x^{n}}{n!}.
\end{equation}
It follows that
\begin{align*}
\frac{\alpha_{k}}{t_{n}^{k/2}}
&=\sum_{\ell=0}^{k}\frac{\psi^{(k-\ell)}(\theta^{\ast})}{(k-\ell)!}
\sum_{\mathcal{S}_{\ell}}\frac{(-1)^{m_{1}+\cdots+m_{\ell}}}{m_{1}!1!^{m_{1}}m_{2}!2!^{m_{2}}\cdots m_{\ell}!\ell!^{m_{\ell}}}
\\
&\qquad\qquad\qquad
\cdot\prod_{j=1}^{\ell}\left(\frac{w^{2}}{\eta''(\theta^{\ast})}\frac{\eta^{(j+2)}(\theta^{\ast})}{(j+2)(j+1)}\right)^{m_{j}}
\frac{(iw)^{k}}{(\eta''(\theta^{\ast}))^{k/2}},
\nonumber
\end{align*}
which implies that
\begin{align}
\alpha_{k}(w)
&=\sum_{\ell=0}^{k}\frac{\psi^{(k-\ell)}(\theta^{\ast})}{(k-\ell)!}
\sum_{\mathcal{S}_{\ell}}\frac{(-1)^{m_{1}+\cdots+m_{\ell}}}{m_{1}!1!^{m_{1}}m_{2}!2!^{m_{2}}\cdots m_{\ell}!\ell!^{m_{\ell}}}
\nonumber
\\
&\qquad\qquad\qquad
\cdot\prod_{j=1}^{\ell}\left(\frac{1}{\eta''(\theta^{\ast})}\frac{\eta^{(j+2)}(\theta^{\ast})}{(j+2)(j+1)}\right)^{m_{j}}
\frac{(i)^{k}w^{k+2(m_{1}+\cdots+m_{\ell})}}{(\eta''(\theta^{\ast}))^{k/2}}.
\nonumber
\end{align}
By the property of standard normal random variable,
\begin{equation}
\int_{-\infty}^{\infty}w^{m}\cdot\frac{e^{-\frac{w^{2}}{2}}}{\sqrt{2\pi}}dw
=(m-1)!!,
\end{equation}
if $m$ is even and $0$ if $m$ is odd.
Hence, the formula for $a_{k}$ follows.


\subsection{Proof of Proposition \ref{prop:ab} (ii): Derivations of $(b_{k})_{k=1}^{\infty}$}

We recall from the proof of Theorem~3.2.2. and Remark~3.2.5. in \cite{modphi} that for integer-valued random variables $X_n$,
\begin{equation}
\mathbb{P}(X_{n}\geq t_{n}x)
=\frac{e^{-t_{n}F(x)}}{\sqrt{2\pi t_{n}\eta''(\theta^{\ast})}}\frac{1}{1-e^{-\theta^{\ast}}}
\left(\psi(\theta^{\ast})+\frac{b_{1}}{t_{n}}+\frac{b_{2}}{t_{n}^{2}}
+\cdots+\frac{b_{v-1}}{t_{n}^{v-1}}
+O\left(\frac{1}{t_{n}^{v}}\right)\right),
\end{equation}
where 
$b_{k}=\int_{\mathbb{R}}\beta_{2k}(w)\frac{e^{-\frac{w^{2}}{2}}}{\sqrt{2\pi}}dw$,
where
$\beta_{j}(w)$ are given via the expansion:
\begin{equation}
g_{n}(w)=\sum_{k=0}^{2\nu-1}\frac{\beta_{k}(w)}{(t_{n})^{k/2}}
+O(t_{n}^{-v}),
\end{equation}
where
\begin{align*}
g_{n}(w)&:=\sum_{k=0}^{\infty}e^{-k\theta^{\ast}}e^{-k\frac{iw}{\sqrt{t_{n}\eta''(\theta^{\ast})}}}
\cdot
\psi\left(\theta^{\ast}+\frac{iw}{\sqrt{t_{n}\eta''(\theta^{\ast})}}\right)
\\
&\qquad\qquad\qquad\qquad
\cdot
e^{t_{n}\left(\eta\left(\theta^{\ast}+\frac{iw}{\sqrt{t_{n}\eta''(\theta^{\ast})}}\right)
-\eta(\theta^{\ast})-\eta'(\theta^{\ast})\frac{iw}{\sqrt{t_{n}\eta''(\theta^{\ast})}}+\frac{w^{2}}{2t_{n}}\right)}
\\
&=\frac{\psi\left(\theta^{\ast}+\frac{iw}{\sqrt{t_{n}\eta''(\theta^{\ast})}}\right)}{1-e^{-\theta^{\ast}-\frac{iw}{\sqrt{t_{n}\eta''(\theta^{\ast})}}}}
\cdot
e^{t_{n}\left(\eta\left(\theta^{\ast}+\frac{iw}{\sqrt{t_{n}\eta''(\theta^{\ast})}}\right)
-\eta(\theta^{\ast})-\eta'(\theta^{\ast})\frac{iw}{\sqrt{t_{n}\eta''(\theta^{\ast})}}+\frac{w^{2}}{2t_{n}}\right)}.
\nonumber
\end{align*}
Note that
\begin{equation}
\sum_{k=0}^{2\nu-1}\frac{\beta_{k}(w)}{(t_{n})^{k/2}}
+O(t_{n}^{-v})
=\frac{1}{1-e^{-\theta^{\ast}-\frac{iw}{\sqrt{t_{n}\eta''(\theta^{\ast})}}}}
\sum_{k=0}^{2\nu-1}\frac{\alpha_{k}(w)}{(t_{n})^{k/2}}.
\end{equation}
We define 
$f(x)=\frac{1}{1-e^{-\theta^{\ast}}x}$,
and 
$g(x)=e^{-x}$.
Then, 
\begin{equation}
f(g(x))=\sum_{n=0}^{\infty}
\sum_{n=0}^{\infty}\frac{d^{n}}{dx^{n}}f(g(0))\frac{x^{n}}{n!},
\end{equation}
where we can compute that 
$f^{(k)}(x)=e^{-k\theta^{\ast}}k!(1-e^{-\theta^{\ast}}x)^{-k-1}$,
and $g^{(k)}(x)=(-1)^{k}e^{-x}$, and
by Fa\`{a} di Bruno's formula, we get
\begin{equation}
\frac{d^{n}}{dx^{n}}f(g(0))
=\sum_{\mathcal{S}_{n}}\frac{n!e^{-\sum_{j=1}^{n}m_{j}\theta^{\ast}}\left(\sum_{j=1}^{n}m_{j}\right)!
(1-e^{-\theta^{\ast}})^{-\sum_{j=1}^{n}m_{j}-1}}
{m_{1}!1!^{m_{1}}m_{2}!2!^{m_{2}}\cdots m_{n}!n!^{m_{n}}}
\cdot
\prod_{j=1}^{n}(-1)^{j\cdot m_{j}}.
\end{equation}
By the formula for $\alpha_{k}(w)$, we have
\begin{align}
\beta_{k}(w)&=
\sum_{n=0}^{k}
\sum_{\mathcal{S}_{n}}\frac{e^{-\sum_{j=1}^{n}m_{j}\theta^{\ast}}\left(\sum_{j=1}^{n}m_{j}\right)!
(1-e^{-\theta^{\ast}})^{-\sum_{j=1}^{n}m_{j}-1}}
{m_{1}!1!^{m_{1}}m_{2}!2!^{m_{2}}\cdots m_{n}!n!^{m_{n}}}
\cdot
\prod_{j=1}^{n}(-1)^{j\cdot m_{j}}
\\
&\qquad
\cdot
\sum_{\ell=0}^{k-n}\frac{\psi^{(k-n-\ell)}(\theta^{\ast})}{(k-n-\ell)!}
\sum_{\mathcal{S}_{\ell}}\frac{(-1)^{m_{1}+\cdots+m_{\ell}}}{m_{1}!1!^{m_{1}}m_{2}!2!^{m_{2}}\cdots m_{\ell}!\ell!^{m_{\ell}}}
\nonumber
\\
&\qquad\qquad\qquad
\cdot\prod_{j=1}^{\ell}\left(\frac{1}{\eta''(\theta^{\ast})}\frac{\eta^{(j+2)}(\theta^{\ast})}{(j+2)(j+1)}\right)^{m_{j}}
\frac{(i)^{k}w^{k+2(m_{1}+\cdots+m_{\ell})}}{(\eta''(\theta^{\ast}))^{k/2}}.
\nonumber
\end{align}
Hence, the formula for $b_{k}$ follows. 

\subsection{Proof of Proposition \ref{prop:etapsi} (i): Computations of $\eta^{(k)}(\theta^{\ast})$}

We recall that $\eta(\theta)=\nu(x(\theta)-1)$ so that
\begin{equation}
\eta^{(k)}(\theta^{\ast})=\nu x^{(k)}(\theta^{\ast}),\qquad k\in\mathbb{N}.
\end{equation}
Note that
\begin{equation}
x(\theta)=e^{\theta+\Vert h\Vert_{L^{1}}(x(\theta)-1)}
=e^{\theta-\Vert h\Vert_{L^{1}}}e^{\Vert h\Vert_{L^{1}}x(\theta)}.
\end{equation}
Therefore, by Leibniz formula, 
\begin{equation}
x^{(k)}(\theta^{\ast})
=\sum_{\ell=0}^{k}\binom{k}{\ell}e^{\theta^{\ast}-\Vert h\Vert_{L^{1}}}
\frac{d^{\ell}}{d\theta^{\ell}}e^{\Vert h\Vert_{L^{1}}x(\theta)}\bigg|_{\theta=\theta^{\ast}}.
\end{equation}
By Fa\`{a} di Bruno's formula, we get
\begin{equation}
\frac{d^{\ell}}{d\theta^{\ell}}e^{\Vert h\Vert_{L^{1}}x(\theta)}\bigg|_{\theta=\theta^{\ast}}
=\sum_{\mathcal{S}_{\ell}}\frac{\ell!\cdot\Vert h\Vert_{L^{1}}^{m_{1}+\cdots+m_{\ell}}
\cdot e^{\Vert h\Vert_{L^{1}}x(\theta^{\ast})}}{m_{1}!1!^{m_{1}}m_{2}!2!^{m_{2}}\cdots m_{\ell}!\ell!^{m_{\ell}}}
\cdot
\prod_{j=1}^{\ell}(x^{(j)}(\theta^{\ast}))^{m_{j}}.
\end{equation}
Hence, 
\begin{align}
x^{(k)}(\theta^{\ast})
&=e^{\theta^{\ast}-\Vert h\Vert_{L^{1}}}
\frac{d^{k}}{d\theta^{k}}e^{\Vert h\Vert_{L^{1}}x(\theta)}\bigg|_{\theta=\theta^{\ast}}
+\sum_{\ell=0}^{k-1}\binom{k}{\ell}e^{\theta^{\ast}-\Vert h\Vert_{L^{1}}}
\frac{d^{\ell}}{d\theta^{\ell}}e^{\Vert h\Vert_{L^{1}}x(\theta)}\bigg|_{\theta=\theta^{\ast}}
\\
&=e^{\theta^{\ast}-\Vert h\Vert_{L^{1}}}
\Vert h\Vert_{L^{1}}
\cdot e^{\Vert h\Vert_{L^{1}}x(\theta^{\ast})}
\cdot x^{(k)}(\theta^{\ast})
\nonumber
\\
&\qquad
+e^{\theta^{\ast}-\Vert h\Vert_{L^{1}}}
\sum_{\mathcal{T}_{k}}\frac{k!\cdot\Vert h\Vert_{L^{1}}^{m_{1}+\cdots+m_{k-1}}
\cdot e^{\Vert h\Vert_{L^{1}}x(\theta^{\ast})}}{m_{1}!1!^{m_{1}}m_{2}!2!^{m_{2}}\cdots m_{k-1}!(k-1)!^{m_{k-1}}}
\cdot
\prod_{j=1}^{k-1}(x^{(j)}(\theta^{\ast}))^{m_{j}}
\nonumber
\\
&\qquad
+\sum_{\ell=0}^{k-1}\binom{k}{\ell}e^{\theta^{\ast}-\Vert h\Vert_{L^{1}}}
\sum_{\mathcal{S}_{\ell}}\frac{\ell!\cdot\Vert h\Vert_{L^{1}}^{m_{1}+\cdots+m_{\ell}}
\cdot e^{\Vert h\Vert_{L^{1}}x(\theta^{\ast})}}{m_{1}!1!^{m_{1}}m_{2}!2!^{m_{2}}\cdots m_{\ell}!\ell!^{m_{\ell}}}
\cdot
\prod_{j=1}^{\ell}(x^{(j)}(\theta^{\ast}))^{m_{j}},
\nonumber
\end{align}
where $\mathcal{T}_{k}$ denotes the set of $(k-1)$-tuples
of non-negative integers $(m_{1},\ldots,m_{k-1})$ satisfying the constraint
$1\cdot m_{1}+2\cdot m_{2}+3\cdot m_{3}+\cdots+(k-1)\cdot m_{k-1}=k$.
Note that $x(\theta^{\ast})=e^{\theta+\Vert h\Vert_{L^{1}}(x(\theta^{\ast})-1)}$.
Also that $\theta^{\ast}=\arg\max_{\theta\geq 0}\{\theta x-\nu(x(\theta)-1)\}$.
Thus, $x=\nu x'(\theta^{\ast})$, which gives $x'(\theta^{\ast})=\frac{x}{\nu}$ and
$x(\theta^{\ast})=\frac{x}{\nu+\Vert h\Vert_{L^{1}}x}$.

For $k\geq 1$, $x^{(k)}(\theta^{\ast})$ can be computed recursively as:
\begin{align}
x^{(k)}(\theta^{\ast})
&=\frac{x(\theta^{\ast})}{1-\Vert h\Vert_{L^{1}}x(\theta^{\ast})}
\sum_{\mathcal{T}_{k}}\frac{k!\cdot\Vert h\Vert_{L^{1}}^{m_{1}+\cdots+m_{k-1}}}{m_{1}!1!^{m_{1}}m_{2}!2!^{m_{2}}\cdots m_{k-1}!(k-1)!^{m_{k-1}}}
\cdot
\prod_{j=1}^{k-1}(x^{(j)}(\theta^{\ast}))^{m_{j}}
\nonumber
\\
&\quad
+\frac{x(\theta^{\ast})}{1-\Vert h\Vert_{L^{1}}x(\theta^{\ast})}\sum_{\ell=0}^{k-1}\binom{k}{\ell}
\sum_{\mathcal{S}_{\ell}}\frac{\ell!\cdot\Vert h\Vert_{L^{1}}^{m_{1}+\cdots+m_{\ell}}}{m_{1}!1!^{m_{1}}m_{2}!2!^{m_{2}}\cdots m_{\ell}!\ell!^{m_{\ell}}}
\cdot
\prod_{j=1}^{\ell}(x^{(j)}(\theta^{\ast}))^{m_{j}}.
\nonumber
\end{align}

\subsection{Proof of Proposition \ref{prop:etapsi} (ii): Computations of $\psi^{(k)}(\theta^{\ast})$}

Let us recall that
\begin{equation}
\psi(\theta)=e^{\nu\int_{0}^{\infty}[F(s;\theta)-x(\theta)]ds}.
\end{equation}
where $F(\cdot;\theta)$ satisfies:
$F(t;\theta)=e^{\theta+\int_{0}^{t}(F(t-s;\theta)-1)h(s)ds}$.
Let $F^{(k)}(\cdot;\theta)$ denote the $k$-th partial derivative of $F(\cdot;\theta)$
w.r.t. $\theta$. By Fa\`{a} di Bruno's formula, we have
\begin{equation}
\psi^{(k)}(\theta^{\ast})
=\sum_{\mathcal{S}_{k}}\frac{k!\cdot\nu^{m_{1}+\cdots+m_{k}}\cdot\psi(\theta^{\ast})}
{m_{1}!1!^{m_{1}}m_{2}!2!^{m_{2}}\cdots m_{k}!k!^{m_{k}}}
\cdot
\prod_{j=1}^{k}\left(\int_{0}^{\infty}\left[F^{(j)}(s;\theta^{\ast})-x^{(j)}(\theta^{\ast})\right]ds\right)^{m_{j}}.
\end{equation}
Moreover, by Leibniz formula, we can compute that
\begin{equation}
F^{(k)}(t;\theta^{\ast})
=\sum_{\ell=0}^{k}\binom{k}{\ell}e^{\theta^{\ast}-\int_{0}^{t}h(s)ds}
\frac{d^{\ell}}{d\theta^{\ell}}e^{\int_{0}^{t}F(t-s;\theta)h(s)ds}\bigg|_{\theta=\theta^{\ast}}.
\end{equation}
By Fa\`{a} di Bruno's formula, similar to deriving the formulas for $x^{(k)}(\theta^{\ast})$, 
we get \eqref{F:recursive}.


\subsection{Computations of $a_{1}$ and $b_{1}$ for Theorem \ref{LDPThm}}

By Proposition \ref{prop:ab}, we can compute that
\begin{align}
&a_{1}=-\frac{1}{2}\frac{\psi''(\theta^{\ast})}{\eta''(\theta^{\ast})}
+\frac{1}{24}\frac{\psi(\theta^{\ast})\eta^{(4)}(\theta^{\ast})+4\psi'(\theta^{\ast})\eta^{(3)}(\theta^{\ast})}{(\eta''(\theta^{\ast}))^{2}}
-\frac{15}{72}\frac{\psi(\theta^{\ast})(\eta^{(3)}(\theta^{\ast}))^{2}}{(\eta''(\theta^{\ast}))^{3}},
\\
&b_{1}
=-\psi(\theta^{\ast})\frac{1}{2}\frac{e^{-\theta^{\ast}}+e^{-2\theta^{\ast}}}{(1-e^{-\theta^{\ast}})^{2}}\frac{1}{\eta''(\theta^{\ast})}
-\frac{1}{2}\psi''(\theta^{\ast})\frac{1}{\eta''(\theta^{\ast})}
\nonumber
\\
&\qquad\qquad
+
\psi(\theta^{\ast})\left[\eta^{(4)}(\theta^{\ast})\frac{1}{24}\frac{3}{(\eta''(\theta^{\ast}))^{2}}
-\frac{1}{2}(\eta^{(3)}(\theta^{\ast}))^{2}\frac{1}{36}\frac{15}{(\eta''(\theta^{\ast}))^{3}}\right]
\nonumber
\\
&\qquad
+\frac{e^{-\theta^{\ast}}}{1-e^{-\theta^{\ast}}}\frac{\psi'(\theta^{\ast})}{\eta''(\theta^{\ast})}
+\psi'(\theta^{\ast})\eta^{(3)}(\theta^{\ast})\frac{1}{2(\eta''(\theta^{\ast}))^{2}}
-\frac{e^{-\theta^{\ast}}}{1-e^{-\theta^{\ast}}}\eta^{(3)}(\theta^{\ast})\frac{1}{2(\eta''(\theta^{\ast}))^{2}}.
\end{align}
Thus to obtain $a_{1}$ and $b_{1}$, we need to compute $\psi(\theta^{\ast})$, $\psi'(\theta^{\ast})$, $\psi''(\theta^{\ast})$,
$\eta''(\theta^{\ast})$, $\eta^{(3)}(\theta^{\ast})$ and $\eta^{(4)}(\theta^{\ast})$
as provided in Proposition \ref{prop:etapsi}.
We recall that
$\psi(\theta)=e^{\nu\int_{0}^{\infty}[F(s;\theta)-x(\theta)]ds}$.
Thus, 
\begin{align}
&\psi'(\theta^{\ast})=\nu\int_{0}^{\infty}[F^{(1)}(s;\theta^{\ast})-x'(\theta^{\ast})]ds\cdot\psi(\theta^{\ast}),
\\
&\psi''(\theta^{\ast})=\left[\left(\nu\int_{0}^{\infty}[F^{(1)}(s;\theta^{\ast})-x'(\theta^{\ast})]ds\right)^{2}
+\nu\int_{0}^{\infty}[F^{(2)}(s;\theta^{\ast})-x''(\theta^{\ast})]ds\right]\cdot\psi(\theta^{\ast}).
\end{align}
We first recall that:
\begin{equation}\label{F:theta:star}
F(t;\theta^{\ast})=e^{\theta^{\ast}+\int_{0}^{t}(F(t-s;\theta^{\ast})-1)h(s)ds},\qquad t\geq 0.
\end{equation}
By differentiating \eqref{F:theta:star} w.r.t. $\theta^{\ast}$
and applying \eqref{F:theta:star} again, we obtain:
\begin{equation}\label{F:theta:star:2}
F^{(1)}(t;\theta^{\ast})=
\left[1+\int_{0}^{t}F^{(1)}(t-s;\theta^{\ast})h(s)ds\right]F(t;\theta^{\ast}),\qquad t\geq 0.
\end{equation}
Finally by differentiating \eqref{F:theta:star:2} w.r.t. $\theta^{\ast}$
and applying \eqref{F:theta:star:2} again , we obtain:
\begin{equation}
F^{(2)}(t;\theta^{\ast})=
\left[
\int_{0}^{t}F^{(2)}(t-s;\theta^{\ast})h(s)ds
+\left(1+\int_{0}^{t}F^{(1)}(t-s;\theta^{\ast})h(s)ds\right)^{2}\right]F(t;\theta^{\ast}),
\end{equation}
for any $t\geq 0$.
Let us recall that $\eta(\theta^{\ast})=\nu(x(\theta^{\ast})-1)$ and 
\begin{equation}
\eta''(\theta^{\ast})=\nu x''(\theta^{\ast}),
\qquad
\eta^{(3)}(\theta^{\ast})=\nu x^{(3)}(\theta^{\ast}),
\qquad
\eta^{(4)}(\theta^{\ast})=\nu x^{(4)}(\theta^{\ast}).
\end{equation}
We can compute from the formulas in Proposition \ref{prop:etapsi} that
\begin{align}
&x'(\theta)=\frac{x(\theta)}{1-\Vert h\Vert_{L^{1}}x(\theta)},
\qquad
x''(\theta)=\frac{(1+\Vert h\Vert_{L^{1}}x'(\theta))^{2}x(\theta)}{1-\Vert h\Vert_{L^{1}}x(\theta)},
\\
&x^{(3)}(\theta)=\frac{\left(3(1+\Vert h\Vert_{L^{1}}x'(\theta))\Vert h\Vert_{L^{1}}x''(\theta)
+(1+\Vert h\Vert_{L^{1}}x'(\theta))^{3}\right)x(\theta)}{1-\Vert h\Vert_{L^{1}}x(\theta)},
\\
&x^{(4)}(\theta)=\frac{\left(3(\Vert h\Vert_{L^{1}}x''(\theta))^{2}
+3(1+\Vert h\Vert_{L^{1}}x'(\theta))\Vert h\Vert_{L^{1}}x^{(3)}(\theta)\right)x(\theta)}
{1-\Vert h\Vert_{L^{1}}x(\theta)}
\\
&\qquad\qquad\qquad
+\frac{\left(6\Vert h\Vert_{L^{1}}(1+\Vert h\Vert_{L^{1}}x'(\theta))^{2}x''(\theta)+(1+\Vert h\Vert_{L^{1}}x'(\theta))^{4}\right)x(\theta)}{1-\Vert h\Vert_{L^{1}}x(\theta)}.
\nonumber
\end{align}
Thus $x'(\theta^{\ast})$, $x''(\theta^{\ast})$, $x^{(3)}(\theta^{\ast})$ and $x^{(4)}(\theta^{\ast})$
can be computed from $x(\theta^{\ast})$.
Note that $\theta^{\ast}=\arg\max_{\theta\geq 0}\{\theta x-\nu(x(\theta)-1)\}$.
Thus, $x=\nu x'(\theta^{\ast})$, which gives $x'(\theta^{\ast})=\frac{x}{\nu}$ and
$x(\theta^{\ast})=\frac{x}{\nu+\Vert h\Vert_{L^{1}}x}$.
Hence, we get
\begin{align}
&x'(\theta^{\ast})=\frac{x}{\nu},
\qquad
x''(\theta^{\ast})=\left(1+\frac{\Vert h\Vert_{L^{1}}x}{\nu}\right)^{2}\frac{x}{\nu},
\\
&x^{(3)}(\theta^{\ast})=\left(1+\frac{\Vert h\Vert_{L^{1}}x}{\nu}\right)^{3}\left(1+\frac{3\Vert h\Vert_{L^{1}}x}{\nu}\right)\frac{x}{\nu},
\\
&x^{(4)}(\theta^{\ast})=3\left(1+\frac{\Vert h\Vert_{L^{1}}x}{\nu}\right)^{4}\left(\left(\frac{\Vert h\Vert_{L^{1}}x}{\nu}\right)^{2}
+\left(1+\frac{3\Vert h\Vert_{L^{1}}x}{\nu}\right)\frac{\Vert h\Vert_{L^{1}}x}{\nu}\right)\frac{x}{\nu}
\\
&\qquad\qquad\qquad
+\left(1+\frac{\Vert h\Vert_{L^{1}}x}{\nu}\right)^{4}
\left(1+\frac{6\Vert h\Vert_{L^{1}}x}{\nu}\right)\frac{x}{\nu}.
\nonumber
\end{align}

\subsection{Proof of Proposition \ref{Prop:IS}}
\begin{proof}[Proof of Proposition \ref{Prop:IS}]
By change of measure for counting processes, see e.g. \cite{Sokol} 
\begin{align}
\mathbb{E}[1_{\{N_{t}\geq xt\}}]
&=\hat{\mathbb{E}}\left[1_{\{N_{t}\geq xt\}}\frac{d\mathbb{P}}{d\hat{\mathbb{P}}}\bigg|_{\mathcal{F}_{t}}\right]
\\
&=\hat{\mathbb{E}}\left[1_{\{N_{t}\geq xt\}}
\cdot 
e^{\int_{0}^{t}\log\left(\frac{\lambda_{s}}{\hat{\lambda}_{s}}\right)dN_{s}
-\int_{0}^{t}(\lambda_{s}-\hat{\lambda}_{s})ds}\right]
\nonumber
\\
&=\hat{\mathbb{E}}\left[1_{\{N_{t}\geq xt\}}
\cdot 
e^{-(\log\gamma)N_{t}
+(\gamma-1)\int_{0}^{t}\lambda_{s}ds}\right]
\nonumber
\\
&=\hat{\mathbb{E}}\left[1_{\{N_{t}\geq xt\}}
\cdot 
e^{-(\log\gamma)N_{t}
+(\gamma-1)\nu t+(\gamma-1)\int_{0}^{t}\int_{0}^{s-}h(s-u)dN_{u}ds}\right]
\nonumber
\\
&=\hat{\mathbb{E}}\left[1_{\{N_{t}\geq xt\}}
\cdot 
e^{-(\log\gamma)N_{t}
+(\gamma-1)\nu t+(\gamma-1)\int_{0}^{t-}[\int_{u}^{t}h(s-u)ds]dN_{u}}\right]
\nonumber
\\
&=\hat{\mathbb{E}}\left[1_{\{N_{t}\geq xt\}}
\cdot 
e^{-(\log\gamma)N_{t}
+(\gamma-1)\nu t+(\gamma-1)\Vert h\Vert_{L^{1}}N_{t}
-(\gamma-1)\int_{0}^{t}H(t-u)dN_{u}}\right]
\nonumber
\\
&=e^{-tI(x)}\hat{\mathbb{E}}\left[1_{\{N_{t}\geq xt\}}
\cdot 
e^{((\gamma-1)\Vert h\Vert_{L^{1}}-\log\gamma)(N_{t}-xt)-(\gamma-1)\int_{0}^{t}H(t-u)dN_{u}}\right],
\nonumber
\end{align}
where we recall that $H(t)=\int_{t}^{\infty}h(s)ds$ denotes
the right tail of the exciting function.
\end{proof}

\section*{Acknowledgements}
The authors thank the AE and an anonymous referee for helpful suggestions, 
which greatly improves the quality of the paper.
The authors are grateful to Jayaram Sethuraman and S. R. S. Varadhan for helpful comments.
The authors also thank Shujian Liao for the help of numerical results.
Fuqing Gao acknowledges support from NSFC Grant 11571262 and 11731012.
Lingjiong Zhu is grateful to the support from NSF Grant DMS-1613164.


\end{document}